\documentclass[a4paper,12pt,draft]{amsart}
\usepackage[margin=30mm]{geometry}
\usepackage{amssymb, amsmath, amsthm, amscd}

\usepackage[T1]{fontenc}
\usepackage{eucal,mathrsfs,dsfont}
\usepackage{color}
\usepackage{mathtools}

\definecolor{mno}{rgb}{0.5,0.1,0.5}

\newcommand{\R}{\mathds R}

\newcommand{\Pp}{\mathds P}

\newcommand{\Ee}{\mathds E}
\newcommand{\N}{\mathds N}

\newcommand{\C}{\mathds C}
\newcommand{\I}{\mathds 1}

\newcommand{\supp}{\operatorname{supp}}

\newtheorem{theorem}{Theorem}[section]
\newtheorem{lemma}[theorem]{Lemma}
\newtheorem{proposition}[theorem]{Proposition}
\newtheorem{corollary}[theorem]{Corollary}

\theoremstyle{definition}

\newtheorem{remark}[theorem]{Remark}

\renewcommand{\theequation}{\thesection.\arabic{equation}}
\renewcommand{\le}{\leqslant}
\renewcommand{\ge}{\geqslant}
\renewcommand{\leq}{\leqslant}
\renewcommand{\geq}{\geqslant}
\renewcommand{\Re}{\ensuremath{\operatorname{Re}}}
\renewcommand{\Im}{\ensuremath{\operatorname{Im}}}
\newtheorem*{ack}{Acknowledgement}
\begin{document}
\allowdisplaybreaks

\title[Feller Processes]{\bfseries Some Theorems on Feller Processes: Transience, Local Times and Ultracontractivity}

\author{Ren\'{e} L.\ Schilling\qquad\qquad Jian Wang}
\thanks{\emph{R.\ Schilling:} TU Dresden, Institut f\"{u}r Mathematische Stochastik, 01062 Dresden, Germany. \texttt{rene.schilling@tu-dresden.de}}
\thanks{\emph{J.\ Wang:}
School of Mathematics and Computer Science, Fujian Normal
University, 350007, Fuzhou, P.R. China.
\texttt{jianwang@fjnu.edu.cn}}
\date{}

\maketitle

\begin{abstract}
We present sufficient conditions for the transience and the existence of local times of a Feller process, and the ultracontractivity of the associated Feller semigroup; these conditions are sharp for L\'{e}vy
processes. The proof uses a local symmetrization technique and a uniform upper bound for the characteristic function of a Feller process. As a byproduct, we obtain for stable-like processes (in the sense of R.\ Bass) on $\R^d$  with smooth variable index $\alpha(x)\in(0,2)$ a transience criterion in terms of the exponent $\alpha(x)$; if $d=1$ and $\inf_{x\in\R} \alpha(x)\in (1,2)$, then the stable-like process has local times.


\end{abstract}

\section{Background and Main Results}\label{section1}
A \emph{Feller process} $(X_t)_{t\ge0}$ with state space $\R^d$ is a strong
Markov process such that the associated operator semigroup $(T_t)_{t\ge0}$,
$$
    T_tu(x)=\Ee^x(u(X_t)),\qquad u\in C_\infty(\R^d),\; t\ge0,\; x\in\R^d,
$$
($C_\infty(\R^d)$ is the space of continuous functions
vanishing at infinity) enjoys the Feller property, i.e.\ it maps
$C_\infty(\R^d)$ into itself. A semigroup is said to be a
\emph{Feller semigroup}, if $(T_t)_{t\ge0}$ is a
one-parameter semigroup of contraction operators
$T_t:C_\infty(\R^d)\rightarrow C_\infty(\R^d)$ which is strongly
continuous: $\lim_{t\to0}\|T_tu-u\|_\infty=0$ for any $u\in C_\infty(\R^d)$, and has the sub-Markov
property: $0\le T_tu\le 1$ whenever $0\le u\le 1$. The infinitesimal
generator $(A,{D}(A))$ of the semigroup $(T_t)_{t\ge0}$ (or of the
process $(X_t)_{t\ge0}$) is given by the strong limit
$$
    Au:=\lim_{t\to0}\frac{T_tu-u}{t}
$$
on the set ${D}(A)\subset C_\infty(\R^d)$ of all $u\in C_\infty (\R^d)$ for
which the above limit exists with respect to the uniform norm. We will
call $(A,{D}(A))$ \emph{Feller generator} for short.

Let $C_c^\infty(\R^d)$ be the space of smooth functions with compact
support. Under the assumption that the test functions
$C_c^\infty(\R^d)$ are contained in $D(A)$, Ph.\ Courr\`{e}ge
\cite[Theorem 3.4]{Cou} proved that the generator $A$ restricted to
$C_c^\infty(\R^d)$ is a pseudo differential operator,
\begin{equation}\label{symbol}
    Au(x)
    =-p(x,D)u(x)
    :=-\int e^{i \langle x,\xi\rangle}\,p(x,\xi)\,\hat{u}(\xi)\,d\xi,\quad u\in C_c^\infty(\R^d),
\end{equation}
with \emph{symbol} $p:\R^d \times \R^d\to \C$; $\hat{u}$ denotes
the Fourier transform of $u$, i.e.\ $\hat{u}(x)=(2\pi)^{{-d}}\int
e^{-i\langle x,\xi\rangle}u(\xi)\,d\xi$. The symbol $p(x,\xi)$ is locally
bounded in $(x,\xi)$, measurable as a function of $x$, and for every
fixed $x\in\R^d$ it is a continuous negative definite function in
the co-variable. This is to say that it enjoys the following
L\'{e}vy-Khintchine representation,
\begin{equation}\label{sy}
    p(x,\xi)
    =c(x)-i\langle b(x),\xi\rangle+\frac{1}{2}\langle\xi,a(x)\xi\rangle
    +\int\limits_{\mathclap{z\neq 0}}\!\!\big(1-e^{i\langle z,\xi\rangle}+i\langle z,\xi\rangle\I_{\{|z|\le1\}}\big)\,\nu(x,dz),
\end{equation}
where $(c(x),b(x),a(x),\nu(x,dz))_{x\in\R^d}$ are the L\'{e}vy
characteristics:  $c(x)$ is a nonnegative measurable function,
$b(x):=(b_j(x))\in\R^d$ is a measurable function, $a(x):=(a_{jk}(x))\in \R^{d\times d}$ is a nonnegative definite
matrix-valued function, and $\nu(x,dz)$ is a nonnegative, $\sigma$-finite kernel on
$\R^d\times\mathscr{B}(\R^d\setminus\{0\})$ such that for every
$x\in\R^d$, $\int_{z\neq 0}(1\wedge |z|^2)\,\nu(x,dz)<+\infty$. For details and a comprehensive bibliography we refer to the monographs \cite{jac-book} by N.\ Jacob and the survey paper \cite{JS}.

\medskip

The purpose of this paper is to provide criteria for the ultracontractivity of Feller semigroups, the transience and the existence of local times of Feller processes. We will frequently make the following two assumptions on the symbol $p(x,\xi)$:
\begin{equation}\label{assumption}
    \|p(\cdot,\xi)\|_\infty\le c(1+|\xi|^2)
    \qquad\text{and}\qquad
    p(\cdot,0)\equiv0.
\end{equation}
The first condition means that the generator has only bounded `coefficients', see, e.g.\ \cite[Lemma 2.1]{S3} or \cite[Lemma 6.2]{Sch}; the second condition implies that the Feller process is conservative in the sense that the life time of the process is almost surely infinite, see \cite[Theorem
5.2]{S1}.

Recall that a Markov semigroup $(T_t)_{t\ge0}$ is \emph{ultracontractive}, if $\|T_t\|_{1\to \infty}<\infty$ for every $t>0$. A Markov process $(X_t)_{t\ge0}$ is \emph{transient}, if there exists a countable cover $\{A_j\}_{j\ge 1}$ of $\R^d$ such that $\Ee^x\big(\int_0^\infty \I_{A_j}(X_t)\,dt\big)<\infty$ for every $x\in\R^d$ and $j\ge1$. Let $(X_t)_{t\ge0}$ be a Markov process on $\R^d$ and $\mathscr F_t:=\sigma(X_s \,:\, s\leq t)$. If there exists an $(\mathscr F_t)_{t\geq 0}$-adapted nonnegative process $(L(\cdot,t))_{t\ge0}$ such that for any measurable bounded function $f\ge0$,
$$
    \int_0^t f(X_s)\,ds
    =\int_{\R^d} f(x)L(x,t)\,dx \qquad\text{almost surely},
$$
then $(L(\cdot,t))_{t\ge0}$ is called the \emph{local time} of the process.

\medskip

We can now state the main result of our paper.

\begin{theorem}\label{pro1}
    Let $(X_t)_{t\ge0}$ be a Feller process with the generator $(A,{D}(A))$ such that $C_c^\infty(\R^d)\subset D(A)$. Then $A|_{C_c^\infty(\R^d)}=-p(\cdot,D)$ is a pseudo differential operator with symbol $p(x,\xi)$. Assume that the symbol satisfies \eqref{assumption}. 
    
    \smallskip\noindent
    \textup{\bfseries (i)} If
    \begin{equation}\label{pro11}
        \lim_{|\xi|\to \infty}\frac{ \inf\limits_{z\in\R^d}\Re p(z,\xi)}{\log (1+|\xi|)}=\infty,
    \end{equation}
    then the corresponding Feller semigroup $(T_t)_{t\ge0}$ is ultracontractive.

    If $P(t,x,dy)$ is the transition function of $(X_t)_{t\ge0}$, then $P(t,x,dy)$ has a density function $p(t,x,y)$ with respect to Lebesgue measure, and for every $t>0$,
    $$
        \sup_{x,y\in\R^d}p(t,x,y)\le (4\pi)^{-d}\int \exp\biggl(-\frac{{t}}{16}\inf_{z\in\R^d}\Re p(z,\xi)\biggr)\,d\xi.
    $$
    Consequently, the Feller semigroup $(T_t)_{t\ge0}$ has the strong Feller property, i.e.\ for any $f\in B_b(\R^d)$ and $t>0$, $T_tf\in C_b(\R^d)$, where $C_b(\R^d)$ is the space of bounded continuous functions on $\R^d$.

    \smallskip\noindent
    \textup{\bfseries (ii)} If
    \begin{equation}\label{pro12}
        \int_{\{|\xi|\le r\}}\frac{d\xi}{\inf\limits_{z\in\R^d}\Re p(z,\xi)} <\infty\quad\text{for every\ \ } r>0,
    \end{equation}
    then the Feller process $(X_t)_{t\ge0}$ is transient.

    \smallskip\noindent
    \textup{\bfseries (iii)} If
    \begin{equation}\label{local}
        \int_{\R^d}\frac{d\xi}{1+\inf\limits_{z\in\R^d}\Re p(z,\xi)}  <\infty,
    \end{equation}
    then the Feller process $(X_t)_{t\ge0}$ has local times $(L(\cdot,t))_{t\ge0}$ on $L^2(dx \otimes d\Pp)$.
\end{theorem}

For a L\'{e}vy process the symbol $p(x,\xi)$ is just the exponent $\psi(\xi)$ of the characteristic function, cf.\ Section \ref{section2}. Therefore, \eqref{pro11} is the Hartman--Wintner condition for the existence of smooth density functions, see \cite{HW} or \cite{KS} for a recent study; \eqref{pro12} is the classic Chung-Fuchs criterion for the transience of the process, see \cite{CF,POO} or \cite[Section 37]{SA}; \eqref {local} is Hawkes' criterion for the existence of local times of the process, see \cite[Theorems 1 and 3]{Haw} or the earlier related result \cite[Theorem 4]{GK}. This shows that the criteria of Theorem \ref{pro1} are sharp for L\'{e}vy processes.

To derive Theorem \ref{pro1} we will need the following uniform upper bound for the characteristic function of a Feller process, which is interesting in its own right.

\begin{theorem}\label{chartth} Let $(X_t)_{t\ge0}$ be a Feller process with the generator $(A,{D}(A))$ such that $C_c^\infty(\R^d)\subset D(A)$. Then $A|_{C_c^\infty(\R^d)}=-p(\cdot,D)$ is a pseudo differential operator with symbol $p(x,\xi)$. Assume that the symbol satisfies \eqref{assumption}. Then for any $t\ge0$ and every $x,\xi\in\R^d$,
$$
    \Big|\Ee^x\big(e^{i\langle X_t-x,\xi\rangle}\big)\Big|
    \le \exp \biggl[- \frac{t}{16}\inf_{z\in\R^d} \Re p(z,2\xi)\biggr].
$$
\end{theorem}

Note that the estimate from Theorem \ref{chartth} is both natural and trivial for a L\'evy process $(Y_t)_{t\geq 0}$:
$$
    \Big|\Ee^x\big(e^{i\langle Y_t-x,\xi\rangle}\big)\Big|
    = \Big|\Ee^0\big(e^{i\langle Y_t,\xi\rangle}\big)\Big|
    = \big|e^{-t\psi(\xi)}\big|
    = e^{-t \Re\psi(\xi)}.
$$

The remaining part of this paper is organized as follows. In Section \ref{section2} we will study the characteristic function of a Feller process. We first point out that, under some mild additional assumptions on a Feller process, the characteristic function is real if, and only if, the associated symbol is real. Then, we give the proof of Theorem \ref{chartth} by using the local symmetrization technique; this approach may well turn out to be useful for further studies of Feller processes. Section \ref{section3} is devoted to proving Theorem \ref{pro1}. Some examples, including stable-like processes, are presented here to illustrate our results. For the sake of completeness, a few necessary properties and estimates for a Feller process are proved in a simple and self-contained way in the appendix.

\section{Characteristic Functions of Feller Processes}\label{section2}

Before we study the characteristic functions of Feller processes, it is instructive to have a brief look at L\'{e}vy processes which are a particular subclass of Feller processes. Our standard
reference for L\'{e}vy processes is the monograph by K.\ Sato
\cite{SA}. A \emph{L\'{e}vy process} $(Y_t)_{t\ge0}$ is a stochastically
continuous random process with stationary and independent
increments. The characteristic function of a L\'{e}vy process has a
particularly simple structure,
$$
    \Ee^x\big(e^{i\langle Y_t-x,\xi\rangle}\big)
    =\Ee^0\big(e^{i\langle Y_t,\xi\rangle}\big)=e^{-t\psi(\xi)}, \quad x,\xi\in\R^d, t\ge0,
$$
where $\psi:\R^d\to \C$ is a continuous negative definite function, i.e.\ it is given by a L\'evy-Khintchine formula of the form \eqref{sy} with characteristics $(c,b,a,\nu(dz))$ which do not depend on $x$.
A short direct calculation shows that the infinitesimal generator of $Y_t$ is given by
$$
    Lu(x)
    =-\psi(D)u(x)
    :=-\int e^{i \langle x,\xi\rangle}\,\psi(\xi)\,\hat{u}(\xi)\,d\xi,
    \quad u\in C_c^\infty(\R^d).
$$
This means that a L\'{e}vy process is generated by a constant-coefficient pseudo differential operator. The symbol
is given by the characteristic exponent (i.e.\ the logarithm of the characteristic functions) of the L\'{e}vy
process.

This relation is no longer true for general Feller processes.
Since the Feller process $(X_t)_{t\ge0}$ is not spatially homogeneous, the
characteristic function of $X_t,t\ge0$, will now depend on the
starting point $x\in\R^d$, i.e.\ on $\Pp^x$. Therefore, we get a
$(d+1)$-parameter family of \emph{characteristic functions}:
\begin{equation}\label{char1}
    \lambda_t(x,\xi)
    :=e^{-i\langle \xi, x\rangle}T_t(e^{i\langle\xi, \cdot\rangle})(x)
    =\Ee^x\big(e^{i\langle X_t-x,\xi\rangle}\big);
\end{equation}
hence, for every $t\ge0$ and $x\in\R^d$, the function $\xi\mapsto \lambda_t(x,\xi)$ is positive definite. Note that \eqref{char1} is well defined, since the operator $T_t$ extends uniquely to a bounded operator on $B_b(\R^d)$ (the space of bounded measurable functions), cf.\ \cite[Section 3]{S1}. According to \cite[Theorem 1.1]{J1},
for any Schwartz function $u$, we have
\begin{equation}\label{char2}
    T_tu(x)
    =\int e^{i \langle x,\xi\rangle}\hat{u}(\xi)\lambda_t(x,\xi)\,d\xi,
\end{equation}
i.e.\ on the Schwartz space $\mathscr S(\R^d)$ the operator $T_t, t\ge0$, is a
pseudo differential operator with symbol $\lambda_t(x,\xi)$.

If the domain of the Feller generator $A$ is sufficiently rich---e.g.\ if it contains
the space $C_b^2(\R^d)$ of twice differentiable functions with bounded derivatives---,
we know from \cite[Theorem 1.2]{J1} (and \cite[Theorem 3.1]{S2} for the general case) that
\begin{equation}\label{char3}
    \frac{d}{dt}\lambda_t(x,\xi)\bigg|_{t=0}
    =-p(x,\xi), \qquad x,\xi\in\R^d.
\end{equation}
This allows us to interpret the symbol probabilistically as the derivative of the characteristic function of the process.
Since the symbol of $T_t$ is not $e^{-tq(x,\xi)}$, we can only expect that the pseudo differential operator $e^{-tq(x,D)}$ with symbol $e^{-tq(x,\xi)}$ is a reasonably good approximation. Under some mild additional assumptions on $p(x,\xi)$, one of us obtained in \cite[Lemma 2]{S4} the following pointwise estimate
\begin{equation}\label{char4}
    |\lambda_t(x,\xi)-e^{-tq(x,\xi)}|\le C(\xi,\rho)\, t^\rho
\end{equation}
for $t\ge0$, $\rho\in[0,1]$ and $x,\xi\in\R^d$. See also the earlier related paper \cite{Jss}.

\subsection{Characteristic Functions and Symbols}
Recall that $((X_t)_{t\ge0},(\Pp^x)_{x\in\R^d})$ is a \emph{solution to the martingale problem} for the operator $(-p(\cdot,D),C_c^\infty(\R^d))$, if $\Pp^x(X_0=x)=1$ for all $x\in\R^d$, and if for all $f\in C_c^\infty(\R^d)$ the process $(M_t^f,\mathscr F_t)_{t\geq 0}$,
$$
    M_t^f := f(X_t)-\int_0^t(-p(X_s,D))f(X_s)\,ds,
$$
is a local martingale under $\Pp^x$. Here $\mathscr{F}_t=\sigma(X_s\::\:s\le t)$ is the natural filtration of the process $(X_t)_{t\geq 0}$. The martingale problem for $(-p(\cdot,D), C_c^\infty(\R^d))$ is \emph{well posed}, if the finite
dimensional distributions for any two solutions with the same initial distribution coincide.

The following result points out the relations between characteristic functions and the symbol of Feller processes.

\begin{theorem}\label{th21th1}
    Let $(X_t)_{t\ge0}$ be a Feller process with the generator $(A,{D}(A))$ such that $C_c^\infty(\R^d)\subset D(A)$. Then $A|_{C_c^\infty(\R^d)}=-p(\cdot,D)$ is a pseudo differential operator with symbol $p(x,\xi)$. For any $x\in\R^d$ and $t\ge0$, let $\lambda_t(x,\xi)$ be the characteristic function of $(X_t)_{t\ge0}$ given by \eqref{char1}. Assume that the symbol $p(x,\xi)$ satisfies \eqref{assumption}. Then, we have the following statements:

 \smallskip\noindent
    \textup{\bfseries (i)}
    The assertion \eqref{char3} holds; that is, for any $x,\xi\in\R^d$, $$
    \frac{d}{dt}\lambda_t(x,\xi)\bigg|_{t=0}
    =-p(x,\xi).$$

 \smallskip\noindent
    \textup{\bfseries (ii)} If the characteristic function $\lambda_t(x,\xi)$ is real for all $x,\xi\in\R^d$ and $t\ge0$, then the symbol $p(x,\xi)$ is also real.

 \smallskip\noindent
    \textup{\bfseries (iii)} Suppose that the martingale problem for $(-p(\cdot,D), C_c^\infty(\R^d))$ is well posed. If the symbol $p(x,\xi)$ is real, then the characteristic function $\lambda_t(x,\xi)$ is real for all $x,\xi\in\R^d$ and $t\ge0$.
\end{theorem}
\begin{remark}
The statement, that the martingale problem for $(-p(\cdot,D), C_c^\infty(\R^d))$ is well posed, is equivalent to saying that the test functions $C_c^\infty(\R^d)$ are an operator core for the Feller operator $(A,D(A))$, i.e.\ $\overline{A|_{C_c^\infty(\R^d)}}=A$. See Proposition \ref{appendix4} in the appendix for the proof.
\end{remark}

We start with some analytic properties of a symbol $p(x,\xi)$ which satisfies \eqref{assumption}.

\begin{lemma}\label{th21th11}
    Let $(X_t)_{t\ge0}$ be a Feller process with the generator $(A,{D}(A))$ such that $C_c^\infty(\R^d)\subset D(A)$, i.e.\ $A|_{C_c^\infty(\R^d)}=-p(\cdot,D)$ is a pseudo differential operator with symbol $p(x,\xi)$. If the symbol $p(x,\xi)$ satisfies \eqref{assumption}, then the function $x\mapsto p(x,\xi)$ is continuous for every fixed $\xi\in\R^d$, and
\begin{equation}\label{th3333}
    \lim_{r\to0}\sup_{z\in\R^d}\sup_{|\xi|\le r}|p(z,\xi)|=0.
\end{equation}
\end{lemma}

\begin{proof}Since $C_c^\infty(\R^d)\subset D(A)$ and $A|_{C_c^\infty(\R^d)}=-p(\cdot,D)$, the operator $-p(\cdot,D)$ maps $C_c^\infty(\R^d)$ into $C_\infty(\R^d)$. By the assumption \eqref{assumption}, the function $x\mapsto p(x,0)=0$ is continuous. Therefore, the required assertions follow from (the proof of) \cite[Theorem 4.4]{S1}. \end{proof}

\begin{proof}[Proof of Theorem \ref{th21th1}]
\textup{\bfseries (i)} Under more restrictive conditions, the conclusion \eqref{char3} has been shown in \cite[Theorem 1.2]{J1} and \cite[Theorem 3.1]{S2}. The following self-contained proof avoids these technical restrictions.

Every Feller semigroup $(T_t)_{t\ge0}$ has a unique extension onto the space $B_b(\R^d)$ of bounded Borel measurable functions, cf.\ \cite[Section 3]{S1}. For notational simplicity, we use $(T_t)_{t\ge0}$ for this extension. According to \cite[Corollary 3.3 and Theorem 4.3]{S1} and Lemma \ref{th21th11}, $t\mapsto T_tu$ is continuous with respect to locally uniform convergence for all continuous and bounded functions $u\in C_b(\R^d)$.

Let $e_\xi(x)=e^{i\langle \xi,x\rangle}$ for $x,\xi\in\R^d$.
By Proposition \ref{appendix1} in the appendix, we know that for $t>0$ and $x,\xi\in\R^d$,
\begin{equation*}\label{proofeee}
    T_te_{\xi}(x)
    =e_{\xi}(x)+\int_0^t T_s A e_{\xi}(x) \,ds.
\end{equation*}
Note that, see e.g.\ \cite[Proof of Lemma 6.3, Page 607, Lines 14--15]{S3},
$$
    -p(x,\xi)=e_{-\xi}(x)Ae_{\xi}(x).
$$
Therefore,
\begin{equation*}\label{proof1111}
    \lambda_t(x,\xi)
    =e_{-\xi}(x)T_te_{\xi}(x)
    =1-e^{-i\langle \xi, x\rangle}\int_0^t T_s\Bigl(p(\cdot,\xi)e^{i\langle \xi, \cdot\rangle}\Bigr)(x)\, ds.
\end{equation*}
Since $ \lambda_0(x,\xi)=1$, we obtain that for any $x,\xi\in\R^d$,
\begin{equation*}\label{proof1111}\begin{aligned}
   \frac{d}{dt}\lambda_t(x,\xi)\bigg|_{t=0}
    &=\lim_{t\to 0} \frac{\lambda_t(x,\xi)-1}{t}\\
    &=-e^{-i\langle \xi, x\rangle}\lim_{t\to 0}\frac{\int_0^t T_s\bigl(p(\cdot,\xi)e^{i\langle \xi, \cdot\rangle}\bigr)(x)\, ds}{t}\\
    &=-e^{-i\langle \xi, x\rangle}p(x,\xi)e^{i\langle \xi, x\rangle} \\
    &=-p(x,\xi).
\end{aligned}\end{equation*}
In the third equality we have used the fact that for fixed $x,\xi\in\R^d$, the function $t\mapsto T_t\bigl(p(\cdot,\xi)e^{i\langle \xi, \cdot\rangle}\bigr)(x)$ is continuous, cf.\ the remark in the last paragraph. This proves (i).

\smallskip\noindent
\textup{\bfseries (ii)} This follows directly from (i).

\smallskip\noindent
\textup{\bfseries (iii)}
For every $t\ge0$ we define $\widetilde{X}_t=2X_0-X_t$. Clearly, $(\widetilde{X}_t)_{t\ge0}$ is also a strong Markov process with the same starting point as $({X}_t)_{t\ge0}$. Let $\widetilde{\Pp}^x$ be the probability of the process $(\widetilde{X}_t)_{t\ge0}$ with starting point $x\in\R^d$, and denote by  $(\widetilde{T}_t)_{t\ge0}$ the semigroup of $(\widetilde{X}_t)_{t\ge0}$. We claim that $(\widetilde{X}_t)_{t\ge0}$ enjoys the Feller property.

Let $P(t,x,dy)$ be the transition function of the process $(X_t)_{t\ge0}$. Then, for all $u\in C_\infty(\R^d)$ and for a fixed $x_0\in\R^d$, we find
\begin{align*}
    |&\widetilde{\Ee}^x(u(\widetilde{X}_t))-\widetilde{\Ee}^{x_0}(u(\widetilde{X}_t))|\\
    &=|\Ee^x(u(2x-X_t))-\Ee^{x_0}(u(2x_0-X_t))|\\
    &\le |\Ee^x(u(2x-X_t))-\Ee^{x}(u(2x_0-X_t))| + |\Ee^x(u(2x_0-X_t))-\Ee^{x_0}(u(2x_0-X_t))|\\
    &\le \int |u(2x-y)-u(2x_0-y)|\,P(t,x,dy) + |\Ee^x(u(2x_0-X_t))-\Ee^{x_0}(u(2x_0-X_t))|.
\end{align*}
Since $u$ is uniformly continuous, we find for every $\epsilon>0$ some $\delta:=\delta(\varepsilon)>0$ such that $|u(z_1)-u(z_2)|<\epsilon$ for all $|z_1-z_2|<\delta$. This and the Feller property of $X_t$ show that for all $|x-x_0|<\delta$,
\begin{align*}
    |\widetilde{\Ee}^x(u(\widetilde{X}_t))-\widetilde{\Ee}^{x_0}(u(\widetilde{X}_t))|
    &\leq \epsilon + |\Ee^x(u(2x_0-X_t))-\Ee^{x_0}(u(2x_0-X_t))|\\
    &\xrightarrow{\quad x\to x_0\quad}\epsilon\xrightarrow{\quad\epsilon\to 0\quad}0.
\end{align*}
On the other hand, let $\tau_{B(x,r)}$ be the first exit time of the process from the ball $B(x,r)$. According to Proposition \ref{appendix2} in the appendix, we know for all $r>0$ and $x\in\R^d$,
$$
    \Pp^x(|X_t-x|\ge r)
    \le \Pp^x(\tau_{B(x,r)}\le t)
    \le c_1\,t\,\sup_{z\in\R^d}\sup_{|\xi|\le 1/r}|p(z,\xi)|
$$
for some constant $c_1>0$.  By the assertion \eqref{th3333} in Lemma \ref{th21th11}, we can choose $\delta_1:=\delta_1(\varepsilon)$ such that
$$
    \Pp^x(|X_t-x|\ge \delta_1)\le \varepsilon/(2\|u\|_\infty).
$$
Since $u\in C_\infty(\R^d)$, we find $\delta_2:=\delta_2(\varepsilon)>0$ such that $\sup_{|z|\ge \delta_2}|u(z)|\le \varepsilon/2$. Therefore, for all $t>0$ and $x\in\R^d$ with $|x|\ge \delta_1+\delta_2$, we have
\begin{align*}
    |\widetilde{\Ee}^x(u(\widetilde{X}_t))|
    &=|\Ee^x(u(2x-X_t))|\\
    &\le \Ee^x\big(|u(2x-X_t)|\I_{\{|X_t-x|\le \delta_1\}}\big)+\Ee^x\big(|u(2x-X_t)|\I_{\{|X_t-x|\ge \delta_1\}}\big)\\
    &\le \sup_{|z|\ge |x|-\delta_1}|u(z)|+\|u\|_\infty \Pp^x(|X_t-x|\ge \delta_1)\\
    &\le \sup_{|z|\ge \delta_2}|u(z)|+\frac{\varepsilon}{2}\le \varepsilon,
\end{align*}
which proves the Feller property of $(\widetilde{T}_t)_{t\ge0}$.

Let $(\widetilde{A}, D(\widetilde{A}))$ be the generator of the Feller semigroup $(\widetilde{T}_t)_{t\ge0}$. We claim that $C_c^\infty(\R^d) \subset D(\widetilde{A})$ and
$$
    \widetilde{A}|_{C_c^\infty(\R^d)}
    =-p(\cdot,D)
    ={A}|_{C_c^\infty(\R^d)}.
$$
For this, we use the weak infinitesimal operator $(\widetilde{A}_w, D(\widetilde{A}_w))$ of the Feller process $(\widetilde X_t)_{t\geq 0}$, see \cite[Chapter I, Section 6]{Dy} for details on the weak infinitesimal operator of a Markov semigroup. According to \cite[Lemma 31.7, Page 209]{SA}, we have $(\widetilde{A},D(\widetilde{A}))=(\widetilde{A}_w, D(\widetilde{A}_w))$. Therefore, it suffices to verify that the test functions $C_c^\infty(\R^d)$ are in the weak domain $D(\widetilde{A}_w)$ and that $\widetilde{A}_w|_{C_c^\infty(\R^d)}=-p(\cdot,D)$. We have to show that for $u\in C_c^\infty(\R^d)$ and every $x\in\R^d$
$$
    \lim_{t\to0}\frac{\widetilde{T}_tu(x)-u(x)}{t}=-p(x,D)u(x).
$$
This can be seen from the following arguments. Using the Fourier transform, the Fubini theorem and the definition of $(\widetilde{X}_t)_{t\ge0}$, we get
\begin{align*}
    \lim_{t\to 0} &\frac{\widetilde{\Ee}^x\Bigl(u\big(\widetilde{X}_{_{t}}\big)-u(x)\Bigr)}{t}\\
    &=\lim_{t\to 0} \frac{1}{t}\widetilde{\Ee}^x\biggl(\int e^{i\langle\widetilde{X}_t,\xi\rangle}\,\hat{u}(\xi)\,d\xi-\int e^{i\langle x,\xi\rangle}\,\hat{u}(\xi)\,d\xi\biggr)\\
    &=\lim_{t\to 0} \frac{1}{t}\int\hat{u}(\xi)\, \widetilde{\Ee}^x\bigl(e^{i\langle\widetilde{X}_t,\xi\rangle}-e^{i\langle x,\xi\rangle}\bigr)\,d\xi\\
    &=\lim_{t\to 0} \frac{1}{t}\int\hat{u}(\xi)e^{i\langle 2x,\xi\rangle}\,{\Ee}^x\Bigl(e^{-i\langle{X}_t,\xi\rangle}-e^{-i\langle x,\xi\rangle}\Bigr)\,d\xi\\
    &=\lim_{t\to 0} \frac{1}{t}\int\hat{u}(\xi)e^{i\langle 2x,\xi\rangle}\,\Ee^x\biggl(\int_0^t Ae^{-i\langle X_s,\xi\rangle}\,ds\biggr)\,d\xi \\
    &=\lim_{t\to 0} \frac{1}{t}\int\hat{u}(\xi)e^{i\langle 2x,\xi\rangle}\, \Ee^x\biggl(\int_0^t e^{-i\langle X_s,\xi\rangle}\big(-p(X_s,-\xi)\big)\,ds \biggr)\,d\xi\\
    &=-\int e^{i\langle x,\xi\rangle}\,p(x,-\xi)\,\hat{u}(\xi)\,d\xi\\
    &=-\int e^{i\langle x,\xi\rangle}\,p(x,\xi)\,\hat{u}(\xi)\,d\xi\\
    &=-p(x,D)u(x).
\end{align*}
In this calculation we have (repeatedly) used that $C_c^\infty(\R^d)\subset{{D}}(A)$, $A|_{C_c^\infty(\R^d)}=-p(\cdot,D)$ and that the function $x\mapsto e_{\xi}(x)=e^{-i\langle x,\xi\rangle}$ belongs for every fixed $\xi\in\R^d$ to the extended domain of the Feller operator $\widetilde{{D}}(A)$, see Proposition \ref{appendix1} below. In the penultimate line we used that $p(x,\xi)$ is real, i.e.\ $p(x,\xi)=p(x,-\xi)$. Therefore, the weak infinitesimal operator of $(\widetilde{T}_t)_{t\ge0}$ on $C_c^\infty(\R^d)$ is just $-p(\cdot,D)$.

According to \cite[Chapter 4, Proposition 1.7]{EK} and \cite[Chapter I, (1.49), Page 40]{Dy}, both $((X_t)_{t\ge0},(\Pp^x)_{x\in\R^d})$ and $((\widetilde{X}_t)_{t\ge0},(\widetilde{\Pp}^x)_{x\in\R^d})$ are solutions to the martingale problem for the operator $(-p(\cdot,D),
C_c^\infty(\R^d))$. Since the martingale problem for $(-p(\cdot,D),C_c^\infty(\R^d))$ is well posed, $((\widetilde{X}_t)_{t\ge0},(\widetilde{\Pp}^x)_{x\in\R^d})$ and $((X_t)_{t\ge0},(\Pp^x)_{x\in\R^d})$ have the same finite-dimensional distributions. In particular, for any $t>0$ and $x,\xi\in\R^d$,
$$
    \Ee^x\big(e^{i\langle X_t-x,\xi\rangle}\big)
    =\widetilde{\Ee}^x\big(e^{i\langle \widetilde{X_t}-x,\xi\rangle}\big),
$$
which shows that $\lambda_t(x,\xi)=\lambda_t(x,-\xi)=\overline{\lambda_t(x,\xi)}$, i.e.\ the characteristic function $\lambda_t(x,\xi)$ is real.
\end{proof}

\subsection{Uniform Upper Bound for Characteristic Functions}\label{section22}
We begin with a uniform upper bound for characteristic functions for small $t\ll 1$.
\begin{proposition}\label{th2}
    Let $(X_t)_{t\ge0}$ be a Feller process with the generator $(A,{D}(A))$ such that $C_c^\infty(\R^d)\subset D(A)$. Then $A|_{C_c^\infty(\R^d)}=-p(\cdot,D)$ is a pseudo differential operator with symbol $p(x,\xi)$. Assume that the symbol satisfies \eqref{assumption} as well as the following sector condition: there exists some $c\in[0,1)$ such that for all $\xi\in\R^d$,
\begin{equation}\label{th22}
    \sup_{x\in\R^d}|\Im p(x,\xi)|\le c \inf_{x\in\R^d}\Re p(x,\xi).
\end{equation}
    Then, for any $\xi\in\R^d$ and $\varepsilon\in (0,1-c)$, there exists some $t_0:=t_0(\xi, \varepsilon)>0$ such that for all $t\in[0,t_0]$
\begin{equation}\label{th21}
    \sup_{x\in \R^d} |\lambda_t(x,\xi)|\le \exp \Bigl[-{(1-c-\varepsilon)\,t}\inf_{z\in\R^d}\Re p(z,\xi)\Bigr].
\end{equation}
\end{proposition}
As a direct consequence of Proposition \ref{th2}, we get

\begin{corollary}\label{th1}
    Let $(X_t)_{t\ge0}$ be a Feller process with generator $(A,D(A))$ satisfying the assumptions of Proposition \ref{th2}. Assume further that the symbol $p(x,\xi)$ is real. Then, for any $\xi\in\R^d$ and $\delta\in(0,1)$, there exists some $t_0:=t_0(\xi, \delta)>0$ such that for any $t\in[0,t_0]$,
\begin{equation}\label{th11}
    \sup_{x\in\R^d} |\lambda_t(x,\xi)|
    \le \exp \biggl[-\delta\,{t}\inf_{z\in\R^d}p(z,\xi)\biggr].
\end{equation}
\end{corollary}

\begin{proof}[Proof of Proposition \ref{th2}]
Fix $\xi\in\R^d$ and $\varepsilon\in(0,1-c)$. Without loss of generality, we may assume that $\inf_{z\in\R^d}\Re p(z,\xi)>0$ and $\xi \neq 0$; otherwise, the assertion \eqref{th21} would  be trivial.

\noindent
\emph{Step 1.} Denote by $P(t,x,dy)$ the transition function of the Feller process $(X_t)_{t\ge0}$ and write $e_{\xi}(x)=e^{i\langle \xi, x\rangle}$ for $x,\xi\in\R^d$. Below we will examine the technique in the proof of Theorem \ref{th21th1} (i) in detail.
Since there exists a constant $c>0$ such that $|p(x,\xi)|\le c(1+|\xi|^2)$ for all $x,\xi\in\R^d$, the Feller operator $A$ has an extension such that $Ae_\xi$ is well defined, cf.\ \cite[Lemma 2.3]{S3} or Proposition \ref{appendix1} in the appendix. The assumption $p(\cdot,0)\equiv 0$ guarantees, see \cite[Theorem 5.2]{S1}, that the process $(X_t)_{t\geq 0}$ is conservative. Therefore, see \cite[Corollary 3.6]{S3} or Proposition \ref{appendix1}, we find for $t>0$ and $x,\xi\in\R^d$,
\begin{equation}\label{proofeee}
    T_te_{\xi}(x)
    =e_{\xi}(x)+\int_0^t T_s A e_{\xi}(x) \,ds.
\end{equation}
Note that, see e.g.\ \cite[Proof of Lemma 6.3, Page 607, Lines 14--15]{S3},
$$
    -p(x,\xi)=e_{-\xi}(x)Ae_{\xi}(x).
$$
Therefore,
\begin{equation}\label{proof1111}\begin{aligned}
    \lambda_t(x,\xi)
    &=e_{-\xi}(x)T_te_{\xi}(x)\\
    &=1-e^{-i\langle \xi, x\rangle}\int_0^t T_s\Bigl(p(\cdot,\xi)e^{i\langle \xi, \cdot\rangle}\Bigr)(x)\, ds\\
    &=1-\int_0^t\int p(y,\xi)e^{i\langle y-x,\xi\rangle}\,P(s,x,dy)\,ds.
\end{aligned}\end{equation}

\noindent
\emph{Step 2.}
Denote by $\Re z$ and $\Im z$ the real and imaginary part of $z\in\C$. From \eqref{char1} we get
\begin{equation}\label{proof222}
    \Re \lambda_t(x,\xi)
    =\int\cos \langle y-x,\xi\rangle\,P(s,x,dy),
\end{equation}
and
$$
    \Im \lambda_t(x,\xi)
    =\int\sin \langle y-x,\xi\rangle\,P(s,x,dy).
$$

Using \eqref{proof1111}, we find for all $t>0$ and $x,\xi\in\R^d$,
$$
    \Re \lambda_t(x,\xi)
    =1-\int_0^t \int \Bigl(\cos\langle y-x,\xi\rangle \Re p(y,\xi)-\sin\langle y-x,\xi\rangle\Im p(y,\xi)\Bigr) P(s,x,dy)\,ds.
$$
Thus, for every $t>0$,
\begin{align*}
    \Re \lambda_t(x,\xi)
    &\ge 1-\int_0^t \int \Bigl(\Re p(y,\xi)+|\Im p(y,\xi)|\Bigr)\,P(s,x,dy)\,ds\\
    &\ge 1-2\sup_{z\in\R^d} |p(z,\xi)|\, t.
\end{align*}
For every $\varepsilon\in (0,1-c)$ we define $t_1=t_1(\xi, \varepsilon)>0$ by
\begin{equation}\label{remark1}
    t_1 :=
    \frac{\varepsilon }{8\sup\limits_{z\in\R^d} |p(z,\xi)|}.
\end{equation}
Then we find for all $t\in(0,t_1]$,
\begin{equation}\label{proofth21}
    \Re \lambda_t(x,\xi)\ge 1-\frac{\varepsilon}{4}.
\end{equation}

Set
$$
    g_1(\xi,\varepsilon)
    :=\frac{\varepsilon}{4|\xi|}\biggl[\frac{\inf_{z\in\R^d} \Re p(z,\xi)}{1+\sup_{z\in\R^d}|\Im p(z,\xi)|}\wedge 1\biggr]
$$
and denote by $\tau_{B(x,r)}$ the first exit time of the process from the open ball $B(x,r)$, i.e.\
$$
    \tau_{B(x,r)}:=\inf\{t>0\::\:X_t\notin B(x,r)\}.
$$
Then we have for every $t>0$ and $x,\xi\in \R^d$,
\begin{align*}
    \Re \lambda_t(x,\xi)&\le 1-\int_0^t\int_{\{|y-x|\le g_1(\xi,\varepsilon)\}} \Re p(y,\xi)\cos \langle y-x,\xi\rangle\,P(s,x,dy)\,ds\\
    &\qquad+\int_0^t\int_{\{|y-x|\le g_1(\xi,\varepsilon)\}} |\Im p(y,\xi)||\sin\langle y-x,\xi\rangle|\,P(s,x,dy)\,ds\\
    &\qquad+2\sup_{z\in\R^d}|p(z,\xi)|\int_0^t \Pp^x\big(|X_s-x|\ge g_1(\xi,\varepsilon)\big)\,ds\\
    &\le 1-\inf_{z\in\R^d}\Re p(z,\xi)\int_0^t\int_{\{|y-x|\le g_1(\xi,\varepsilon)\}} \cos \langle y-x,\xi\rangle\,P(s,x,dy)\,ds\\
    &\qquad+\sup_{z\in\R^d} |\Im p(z,\xi)|\int_0^t\int_{\{|y-x|\le g_1(\xi,\varepsilon)\}} |\sin \langle y-x,\xi\rangle|\,P(s,x,dy)\,ds\\
    &\qquad+2\sup_{z\in\R^d}|p(z,\xi)|\int_0^t \Pp^x\big(\tau_{B(x,g_1(\xi,\varepsilon))}\le s\big)\,ds.
\end{align*}
In the second inequality we used that $\cos \langle y-x,\xi\rangle \geq 0$ on the set $\{|y-x|\le g_1(\xi,\varepsilon)\}$ and
$\{|X_s-x|\ge g_1(\xi,\varepsilon)\} \subset \{\tau_{B(x,g_1(\xi,\varepsilon))}\le
s\}$.

We know from \cite[Lemmas 4.1 and Lemma 5.1]{S3}, see also Proposition \ref{appendix2} in the appendix for a simple self-contained proof, that for $x,\xi\in\R^d$ and $s>0$,
\begin{equation}\label{exit}\begin{aligned}
    \Pp^x\big(\tau_{B(x,g_1(\xi,\varepsilon))}\le s\big)
    &\le c_1\,s \sup_{|y-x|\le g_1(\xi,\varepsilon)} \; \sup_{|\eta|\le 1/g_1(\xi,\varepsilon)}|p(y,\eta)|\\
    &\le c_1\,s \sup_{z\in\R^d} \; \sup_{|\eta|\le 1/g_1(\xi,\varepsilon)}|p(z,\eta)|
\end{aligned}\end{equation}
for some absolute constant $c_1>0$. Note that on the set $\{|y-x|\le g_1(\xi,\varepsilon)\}$,
$$
    |\sin \langle y-x,\xi\rangle|\le |\langle y-x,\xi\rangle|\le g_1(\xi,\varepsilon)|\xi|.
$$
If we combine all estimates from above, we arrive at
\begin{align*}
    \Re \lambda_t(x,\xi)&\le 1-\inf_{z\in\R^d} \Re p(z,\xi)\int_0^t\int_{\{|y-x|\le g_1(\xi,\varepsilon)\}} \cos\langle y-x,\xi\rangle\,P(s,x,dy)\,ds\\
    &\qquad +\frac{\varepsilon}{4}\inf_{z\in\R^d}\Re p(z,\xi) \, t+{c_1} \sup_{z\in\R^d}|p(z,\xi)| \; \sup_{z\in\R^d}\sup_{|\eta|\le1/g_1(\xi,\varepsilon)}|p(z,\eta)|\, t^2\\
    &\le 1-\inf_{z\in\R^d} \Re p(z,\xi)\int_0^t\int \cos \langle y-x,\xi\rangle\,P(s,x,dy)\,ds\\
    &\qquad +\inf_{z\in\R^d} \Re p(z,\xi)\int_0^t\Pp^x\big(|X_s-x|\ge g_1(\xi,\varepsilon)\big)\,ds\\
    &\qquad +\frac{\varepsilon}{4}\inf_{z\in\R^d}\Re p(z,\xi)\, t +{c_1} \sup_{z\in\R^d}|p(z,\xi)| \; \sup_{z\in\R^d}\sup_{|\eta|\le1/g_1(\xi,\varepsilon)}|p(z,\eta)|\,t^2\\
    &\le 1-\inf_{z\in\R^d} \Re p(z,\xi)\int_0^t\int \cos \langle y-x,\xi\rangle\,P(s,x,dy)\,ds\\
    &\qquad +\frac{c_1}{2} \inf_{z\in\R^d} \; \Re p(z,\xi)\sup_{z\in\R^d}\sup_{|\eta|\le1/g_1(\xi,\varepsilon)}|p(z,\eta)|\,t^2\\
    &\qquad +\frac{\varepsilon}{4}\inf_{z\in\R^d}\Re p(z,\xi)\, t +c_1 \sup_{z\in\R^d}|p(z,\xi)|\; \sup_{z\in\R^d}\sup_{|\eta|\le1/g_1(\xi,\varepsilon)}|p(z,\eta)|\,t^2\\
    &= 1-\inf_{z\in\R^d} \Re p(z,\xi)\int_0^t \Re \lambda_s(x,\xi)\,ds \\
    &\qquad +\frac{\varepsilon}{4}\inf_{z\in\R^d}\Re p(z,\xi)\, t +\frac{3c_1}{2} \sup_{z\in\R^d}|p(z,\xi)|\; \sup_{z\in\R^d}\sup_{|\eta|\le1/g_1(\xi,\varepsilon)}|p(z,\eta)|\,t^2,
\end{align*}
For the third inequality we used $\{|X_s-x|\ge g_1(\xi,\varepsilon)\} \subset \{\tau_{B(x,g_1(\xi,\varepsilon))}\le
s\}$ and \eqref{exit}, while the last equality follows from \eqref{proof222}.

Using \eqref{proofth21} we find for all $t\in (0,t_1]$,
\begin{equation}\label{proofth22}\begin{aligned}
    \Re \lambda_t(x,\xi)
    &\le 1-\Bigl(1-\frac{\varepsilon}{2}\Bigr)\inf_{z\in\R^d} \Re p(z,\xi)\,t\\
    &\qquad +\frac{3c_1}{2} \sup_{z\in\R^d}|p(z,\xi)| \; \sup_{z\in\R^d}\sup_{|\eta|\le1/g_1(\xi,\varepsilon)}|p(z,\eta)|t^2.
\end{aligned}\end{equation}

\noindent
\emph{Step 3.}
We will now consider $\Im \lambda_t(x,\xi)$. For every $t>0$ and $x,\xi\in\R^d$, we find from \eqref{proof1111} that
$$
    \Im \lambda_t(x,\xi)
    =-\int_0^t \int \Bigl(\cos\langle y-x,\xi\rangle \Im p(y,\xi)+\sin\langle y-x,\xi\rangle\Re p(y,\xi)\Bigr) P(s,x,dy)\,ds.
$$
Therefore, for each $t>0$,
\begin{align*}
    \big|\Im \lambda_t(x,\xi)\big|
    &\le \sup_{z\in\R^d}|\Im p(z,\xi)|\,t + \int_0^t \int\Re p(y,\xi) |\sin \langle y-x,\xi\rangle|\, P(s,x,dy)\,ds.
\end{align*}
Set
$$
    g_2(\xi,\varepsilon)
    :=\frac{\varepsilon}{4|\xi|} \, \frac{\inf_{z\in\R^d}\Re p(z,\xi)}{\sup_{z\in\R^d}\Re p(z,\xi)}.
$$
Then, similar to the reasoning in Step 2, we see
\begin{align*}
    \big|\Im \lambda_t(x,\xi)\big|
    &\le\sup_{z\in\R^d}|\Im p(z,\xi)|\,t\\
    &\quad +\sup_{z\in\R^d}\Re p(z,\xi)\int_0^t \int_{\{|y-x|\le g_2(\xi,\varepsilon)\}} |\sin \langle y-x,\xi\rangle|\, P(s,x,dy)\,ds\\
    &\quad+ \sup_{z\in\R^d}\Re p(z,\xi)\int_0^t\Pp^x\big(|X_s-x|\ge g_2(\xi,\varepsilon)\big)\,ds\\
    &\le \sup_{z\in\R^d}|\Im p(z,\xi)|\,t+ \frac{\varepsilon}{4}\inf_{z\in\R^d}\Re p(z,\xi) \,t \\
    &\quad + \frac{c_1}{2} \sup_{z\in\R^d}|p(z,\xi)|\; \sup_{z\in\R^d}\sup_{|\eta|\le1/g_2(\xi,\varepsilon)}|p(z,\eta)|\,t^2.
\end{align*}
This along with \eqref{proofth21} and \eqref{proofth22} yields for all $t\in (0,t_1]$,
\begin{align*}
    |\lambda_t (x,\xi)|&\le  |\Re \lambda_t(x,\xi)| +|\Im \lambda_t(x,\xi)| \\
    &= \Re \lambda_t(x,\xi) +|\Im \lambda_t(x,\xi)| \\
    & \le 1-\biggl[\Bigl(1-\frac{3\varepsilon}{4}\Bigr)\inf_{z\in\R^d} \Re p(z,\xi)-\sup_{z\in\R^d}|\Im p(z,\xi)|\biggr]\, t \\
    &\quad +\frac{c_1}{2} \sup_{z\in\R^d}|p(z,\xi)|\biggl[3\sup_{z\in\R^d}\sup_{|\eta|\le1/g_1(\xi,\varepsilon)}|p(z,\eta)| +\sup_{z\in\R^d}\sup_{|\eta|\le1/g_2(\xi,\varepsilon)}|p(z,\eta)|\biggr]\, t^2.
\end{align*}
Define $t_2=t_2(\varepsilon,\xi)$ by
$$
    t_2 :=
    t_1 \wedge \frac{%
        {\varepsilon}\inf\limits_{z\in\R^d} \Re p(z,\xi)%
        }{%
        2c_1\sup\limits_{z\in\R^d}|p(z,\xi)|\left[3\sup\limits_{z\in\R^d}\:\sup\limits_{|\eta|\le1/g_1(\xi,\varepsilon)}|p(z,\eta)| +\sup\limits_{z\in\R^d}\:\sup\limits_{|\eta|\le1/g_2(\xi,\varepsilon)}|p(z,\eta)|\right]
        }\,.
$$
Then we obtain for all $t\in (0,t_2]$,
\begin{equation}\label{proofth23}
    \big|\lambda_t(x,\xi)\big|
    \le 1-\Bigl[(1-\varepsilon)\inf_{z\in\R^d} \Re p(z,\xi)-\sup_{z\in\R^d}|\Im p(z,\xi)|\Bigr] \,t.
\end{equation}
Because of the sector condition \eqref{th22}, we see
\begin{align*}
    \big|\lambda_t(x,\xi)\big|
    &\le 1-(1-c-\varepsilon)\, t\inf_{z\in\R^d} \Re p(z,\xi)\\
    &\le\exp\biggl[-(1-c-\varepsilon){t}\inf_{z\in\R^d}\Re p(z,\xi)\biggr],
\end{align*}
where the last estimate follows from the elementary inequality $1-r\le
e^{-r}$ for $r\in\R$.
 In particular, for any $t\in(0,t_2]$,
\begin{equation*}
    \sup_{x\in\R^d}|\lambda_t(x,\xi)|
    \le\exp\biggl[-(1-c-\varepsilon)\,t\inf_{z\in\R^d} \Re p(z,\xi)\biggr],
\end{equation*} which is the required assertion by taking $t_0=t_2$.
\end{proof}

\begin{remark}\label{remarkone}
 \smallskip\noindent
    \textup{\bfseries (i)}  Note that $t_2(\varepsilon,\xi)\to 0$ as $\varepsilon\to 0$ which means that the approach above fails for $\varepsilon=0$.
Therefore, Proposition \ref{th2} will, in general, not hold with $\varepsilon=0$ nor can we expect Corollary \ref{th1} to be true if $\delta=1$.

 \smallskip\noindent
    \textup{\bfseries (ii)}  A variant of our approach yields a uniform lower bound for characteristic functions for small $t$. More precisely: \emph{Let $(X_t)_{t\ge0}$ be a Feller process with the generator $(A,{D}(A))$ such that $C_c^\infty(\R^d)\subset D(A)$. Then $A|_{C_c^\infty(\R^d)}=-p(\cdot,D)$ is a pseudo differential operator with symbol $p(x,\xi)$. Assume that the symbol satisfies \eqref{assumption}. Then for any $\varepsilon>0$ and $\xi\in\R^d$, there exists some $t_0=t_0(\varepsilon,\xi)>0$ such that for any $t\in(0,t_0]$
$$
    \inf_{x\in\R^d} |\lambda_t(x,\xi)|\ge \exp \biggl(-{(1+\varepsilon)\,t}\sup_{z\in\R^d} | p(z,\xi)|\biggr).
$$} 

\smallskip\noindent
\textup{\bfseries (iii)} From the pointwise estimate \eqref{char4}, one can get that for any $\xi\in\R^d$, there exists some $t_0:=t_0(\xi)>0$ such that for all $t\in[0,t_0]$,
$$
    \sup_{x\in\R^d} |\lambda_t(x,\xi)|\le e^{-t\,\inf\limits_{x\in\R^d}\Re p(x,\xi)}+ C(\xi,1)\,t.
$$
Although the remainder term $C(\xi,1)$ in \eqref{char4} is well-known, see \cite[Lemma 2]{S4} for details,
we were not able to derive from this estimate the assertion \eqref{th21}.
\end{remark}

Our main result in this subsection is the following uniform upper bound of the characteristic function, which is just Theorem \ref{chartth} in Section \ref{section1}.
\begin{theorem}\label{th23}
    Let $(X_t)_{t\ge0}$ be a Feller process with the generator $(A,{D}(A))$ such that $C_c^\infty(\R^d)\subset D(A)$, i.e.\ $A|_{C_c^\infty(\R^d)}=-p(\cdot,D)$ is a pseudo differential operator with symbol $p(x,\xi)$. For all $x\in\R^d$ and $t\ge0$, let $\lambda_t(x,\xi)$ be the characteristic function of $(X_t)_{t\ge0}$ given by \eqref{char1}. Assume that the symbol satisfies \eqref{assumption}. Then, for all $t\ge0$ and $\xi\in \R^d$,
    $$
        \sup_{x\in\R^d}|\lambda_{t}(x,\xi)|\le \exp \biggl[- \frac{t}{16}\inf_{z\in\R^d} \Re p(z,2\xi)\biggr].
    $$
\end{theorem}

\begin{proof} \emph{Step 1.}
We first assume that the characteristic function $\lambda_t(x,\xi)$ is real for every $t\ge0$ and every $x,\xi\in \R^d$. Then, by Theorem \ref{th21th1} (ii), the corresponding symbol $p(x,\xi)$ is also real. On the other hand, applying Corollary \ref{th1} with $\delta=1/2$ yields that there exists some $t_0:=t_0(\xi)>0$ such that for all $t\in (0,t_0]$, \begin{equation}\label{proof50}
    \sup_{x\in\R^d}|\lambda_t(x,\xi)|
    \le\exp\biggl[-\frac{1}{2}\,t\inf_{z\in\R^d} p(z, \xi)\biggr].
\end{equation}
Since $\sqrt{p(x,\cdot)}$ is subadditive, i.e.\  $\sqrt{ p(x,\xi_1+\xi_2)}\le \sqrt{ p(x,\xi_1)}+ \sqrt{ p(x,\xi_2)}$ for all $x,\xi_1,\xi_2\in\R^d$, we see
$$
    \inf_{z\in\R^d}  p(z, 2\xi) \le 4 \inf_{z\in\R^d}  p(z, \xi).
$$
Thus, \eqref{proof50} leads to
\begin{equation}\label{proof51}
    \sup_{x\in\R^d}|\lambda_t(x,\xi)|
    \le\exp\biggl[-\frac{1}{8}\,t\inf_{z\in\R^d} p(z, 2\xi)\biggr].
\end{equation}

For every $t>0$ we can choose some $m:=m(\xi)\in \N$ such that
$\frac{t}{m}\in (0,t_3]$, where
$$
    t_3= t_0 \wedge \frac{2}{\inf_{z\in\R^d} p(z,2\xi)}.
$$
We will prove by induction that for any $k=1,2, \cdots,m$,
\begin{equation}\label{proof52}\sup_{x\in\R^d}|\lambda_{k \frac tm}(x,\xi)|\le \exp \biggl[- \frac k{16}\,\frac tm\inf_{z\in\R^d} p(z,2\xi)\biggr].\end{equation}

First, according to \eqref{proof51}, we know that \eqref{proof52} holds with $k=1$. Assume that \eqref{proof52} is satisfied with $k=j$. Then, for $k=j+1$, by the Markov property and the fact that the characteristic function $\lambda_t(x,\xi)$ is real for any $t\ge0$ and $x,\xi\in\R^d$, we have
\begin{align*}
   \bigg|\lambda_{(j+1)t/m}(x,\xi)\bigg|&=\bigg|\Ee^x\Bigl(e^{i\langle X_{(j+1)t/m}-x,\xi\rangle}\Bigr)\bigg|\\
   &=\bigg|\Ee^x\Bigl(e^{i\langle X_{t/m}-x,\xi\rangle} \Ee^{X_{t/m}}\Bigl(e^{i\langle X_{jt/m}-x,\xi\rangle}\Bigr)\Bigr)\bigg|\\
   &=\bigg|\Ee^x\Bigl(e^{i\langle X_{t/m}-x,\xi\rangle} \lambda_{jt/m}(X_{t/m},\xi)\Bigr)\bigg|\\
   &=\bigg|\Ee^x\Bigl(\cos\,\langle X_{t/m}-x,\xi\rangle\, \lambda_{jt/m}(X_{t/m},\xi)\Bigr)\bigg|\\
   &\le \sqrt{\Ee^x\Bigl(\cos^2\langle X_{t/m}-x,\xi\rangle \Bigr)}\sqrt{ \Ee^x\Bigl(\lambda_{jt/m}^2(X_{t/m},\xi)\Bigr)}\\
   &\le \sqrt{\Ee^x\Bigl(\cos^2\langle X_{t/m}-x,\xi\rangle \Bigr)}\, \sup_{x\in\R^d} |\lambda_{jt/m}(x,\xi)|\\
   &\le \sqrt{\frac{1+\Ee^x(\cos\,\langle X_{t/m}-x,2\xi\rangle) }{2}} \,  \exp \biggl[- \frac{j}{16}\,\frac tm\inf_{z\in\R^d} p(z,2\xi)\biggr].
\end{align*}
The first inequality follows from the Cauchy-Schwarz inequality, and in the last inequality we have used the induction hypothesis and the fact that
$$
    \cos^2 \theta
    =\frac{1}{2}\big(1+\cos (2\theta)\big),\qquad \theta\in\R.
$$
Therefore,
$$
    \sup_{x\in\R^d} \bigg|\lambda_{(j+1) \frac tm}(x,\xi)\bigg|
    \le \sup_{x\in\R^d} \sqrt{\frac{1+\Ee^x(\cos\,\langle X_{t/m}-x,2\xi\rangle) }{2}}
      \,  \exp \biggl[- \frac{j}{16}\frac tm\inf_{z\in\R^d} p(z,2\xi)\biggr].
$$

For any $x\in\R^d$ we can use \eqref{proof50} and the assumptions that $\lambda_t(x,\xi)$ is real and $\frac{t}{2m}\inf_{z\in\R^d}  p(z,2\xi)\le 1$ to deduce
\begin{align*}
   &\frac{1+\Ee^x(\cos\,\langle X_{t/m}-x,2\xi\rangle) }{2}\\
   &\qquad= \frac{1+\lambda_{t/m} (x,2\xi)}{2}\\
   &\qquad\le \frac{1+|\lambda_{t/m} (x,2\xi)|}{2}\\
   &\qquad\le   \frac{1+\exp\Bigl[-\,\frac{t}{2m}\inf_{z\in\R^d}  p(z,2\xi)\Bigr]}{2}\\
   &\qquad\le \frac{1+1-\Bigl[\,\frac{t}{2m}\inf_{z\in\R^d}  p(z,2\xi)\Bigr]+\frac{1}{2}\Bigl[\,\frac{t}{2m}\inf_{z\in\R^d}  p(z,2\xi)\Bigr]^2 }{2}\\
   &\qquad\le 1-\frac{\,\frac{t}{m}\inf_{z\in\R^d}  p(z,2\xi)}{8}\\
   &\qquad\le \exp\bigg[-\frac{\,\frac{t}{m}\inf_{z\in\R^d}  p(z,2\xi)}{8}\bigg],
\end{align*}
where the third and the last inequality follow from the elementary estimates
$$
    1-r\le e^{-r}\le 1-r+r^2/2,\qquad r\ge0.
$$
Thus, we get
$$
    \sup_{x\in\R^d}\sqrt{\frac{1+\Ee^x(\cos\,\langle X_{t/m}-x,2\xi\rangle) }{2}}
    \le \exp\bigg[-\frac{\,\frac{t}{m}\inf_{z\in\R^d} p(z,2\xi)}{16}\bigg],
$$
and the induction step is complete.

Taking $k=m$ in \eqref{proof52} we find, in particular, for all $t>0$,
\begin{equation}\label{proof522}
    \sup_{x\in\R^d}|\lambda_{t}(x,\xi)|
    \le \exp \biggl[- \frac{t}{16}\inf_{z\in\R^d} p(z,2\xi)\biggr].
\end{equation}

\noindent\emph{Step 2.}
Now we consider the general case where $\lambda_t(x,\xi)$ is not necessarily real. Using a local symmetrization technique we can reduce the general case to the situation treated in Step 1. Let $(X_t)_{t\ge0}$ be a Feller process with the generator $(A,{D}(A))$ and the semigroup $(T_t)_{t\ge0}$ such that $C_c^\infty(\R^d)\subset D(A)$ and $A|_{C_c^\infty(\R^d)}=-p(\cdot,D)$ is a pseudo differential operator with symbol $p(x,\xi)$. Denote by $\lambda_t(x,\xi)$ the characteristic function of $X_t-x$ under $\Pp^x$.

 Construct on the same probability space a stochastic process $(X^*_t)_{t\ge0}$ such that $X^*_0=X_0$ and $(X^*_t)_{t>0}$ is an independent copy of $(X_t)_{t>0}$, and define a further process $(\widetilde{X}_t)_{t\ge0}$ on $\R^d$ by $\widetilde{X}_t=2X^*_0-X^*_t$, $t\ge0$. Clearly, the process $(\widetilde{X}_t)_{t>0}$ is independent of $(X_t)_{t>0}$ but it has the same initial distribution, i.e.\ $\widetilde{X}_0 \sim X_0$. From the proof of Theorem \ref{th21th1} (iii) we see that $(\widetilde{X}_t)_{t\ge0}$ is a Feller process with the generator $(\widetilde{A},{D}(\widetilde{A}))$ and the semigroup $(\widetilde{T}_t)_{t\ge0}$ such that $C_c^\infty(\R^d)\subset D(\widetilde{A})$, and $\widetilde{A}|_{C_c^\infty(\R^d)}=-\widetilde{p}(\cdot,D)$ is a pseudo differential operator with symbol $\widetilde{p}(x,\xi)=p(x,-\xi)$. Moveover, the characteristic function of $(\widetilde{X}_t)_{t\ge0}$ is $\widetilde{\lambda}_t(x,\xi)=\lambda_t(x,-\xi)$ for every $t\ge0$ and $x\in\R^d$.

For every $t\ge0$ we define the local symmetrization ${X}^S_t=\frac{1}{2}\big(X_t+\widetilde{X}_t\big)$. Lemma \ref{symmetric} below shows that the local symmetrization $(X^S_t)_{t\ge0}$ is a Feller process with the generator $(A^S,{D}(A^S))$ such that $C_c^\infty(\R^d)\subset D(A^S)$, and $A^S|_{C_c^\infty(\R^d)}=-p^S(\cdot,D)$ is a pseudo differential operator with symbol $2\Re p(x,\xi/2)$; moreover, the characteristic function of $(X^S_t)_{t\ge0}$ is $|\lambda_t(x,\xi/2)|^2$.

We can now apply the conclusion of Step 1, in particular \eqref{proof522}, to the process $(X^S_t)_{t\ge0}$; we obtain that for any $t>0$,
$$
    \sup_{x\in\R^d}|\lambda_{t}(x,\xi/2)|^2\le \exp \biggl[- \frac{t}{8}\inf_{z\in\R^d} \Re p(z,\xi)\biggr].
$$
That is, $$
    \sup_{x\in\R^d}|\lambda_{t}(x,\xi)|^2\le \exp \biggl[- \frac{t}{8}\inf_{z\in\R^d} \Re p(z,2\xi)\biggr],
$$
which is what we have claimed.
\end{proof}

\begin{lemma}[\bfseries Local Symmetrization]\label{symmetric}
Let $(X_t)_{t\ge0}$ be a Feller process with generator $(A,D(A))$ such that $C_c^\infty(\R^d)\subset D(A)$, i.e.\ $A|_{C_c^\infty(\R^d)}=-p(\cdot,D)$. Denote by $(X^*_t)_{t\ge0}$ an independent copy of $(X_t)_{t\ge0}$, set $\widetilde X_t := 2X_0^*-X^*_t$ and let $(X^S_t)_{t\ge0}$ be the local symmetrization of $(X_t)_{t\geq 0}$, i.e.\ for any $t\ge0$, ${X}^S_t=\frac{1}{2}\big(X_t+\widetilde{X}_t\big)$. Then, $(X^S_t)_{t\ge0}$ is a Feller process with the generator $(A^S,{D}(A^S))$ such that

 \smallskip\noindent
    \textup{\bfseries (i)} $C_c^\infty(\R^d)\subset D(A^S)$, and $A^S|_{C_c^\infty(\R^d)}=-p^S(\cdot,D)$ is a pseudo differential operator with symbol $p^S(x,\xi) = 2\Re p(x,\xi/2)$;

 \smallskip\noindent
    \textup{\bfseries (ii)}
     the characteristic function of $(X^S_t)_{t\ge0}$ is $|\lambda_t(x,\xi/2)|^2$ for every $t\ge0$ and $x\in\R^d$.
\end{lemma}

\begin{proof} Clearly, $({X}^S_t)_{t\ge0}$ is a strong Markov process. Denote by $(T_t^S)_{t\ge0}$ the semigroup of $({X}^S_t)_{t\ge0}$. Since $(X_t)_{t>0}$ and $(\widetilde{X})_{t>0}$ are independent with $X_0 \sim \widetilde{X}_0$ we find all $u\in B_b(\R^d)$ (the set of bounded measurable functions on $\R^d$), $x\in\R^d$ and $t\ge0$,
\begin{equation}\label{symmet1}
    T_t^Su(x)
    =\int u(z)\, P^S(t,x,dz)
    = \iint u\Big(\frac{z_1+z_2}{2}\Big)\, P(t,x,dz_1)\, \widetilde{P}(t,x,dz_2),
\end{equation}
where $P(t,x,dy)$, $\widetilde{P}(t,x,dy)$ and $P^S(t,x,dy)$ are the transition functions of  $({X}_t)_{t\ge0}$, $(\widetilde{{X}}_t)_{t\ge0}$ and $({X}^S_t)_{t\ge0}$, respectively.

As mentioned above, since $p(\cdot, 0)\equiv0$, \cite[Theorem
5.2]{S1} shows that the processes $(X_t)_{t\ge0}$ and $(\widetilde{X}_t)_{t\ge0}$ are conservative. Thus, according to Proposition \ref{appendix5} (i) in the appendix, both semigroups $(T_t)_{t\ge0}$ and $(\widetilde{T}_t)_{t\ge0}$ are $C_b$-Feller semigroups, i.e.\  they map the space $C_b(\R^d)$ of bounded continuous functions on $\R^d$ into itself. This along with \eqref{symmet1} yields the $C_b$-Feller property of $(T_t^S)_{t\ge0}$. Indeed, for any fixed $x_0\in\R^d$ and $t>0$, due to the $C_b$-Feller property of $(T_t)_{t\ge0}$ and $(\widetilde{T}_t)_{t\ge0}$, we know that the probability measures $P(t,x,dz)$ and $\widetilde{P}(t,x,dz )$ converge weakly to $P(t,x_0,dz)$ and $\widetilde{P}(t,x,dz)$ respectively, as $x$ tends to $x_0$. Thus, the convolution of $P(t,x,dz)$ and $\widetilde{P}(t,x,dz )$ converges weakly to the convolution of $P(t,x_0,dz)$ and $\widetilde{P}(t,x_0,dz)$ as $x$ tends to $x_0$. That is, for any $u\in C_b(\R^d)$,
$$\lim_{x\to x_0}  \iint u({z_1+z_2})\, P(t,x,dz_1)\, \widetilde{P}(t,x,dz_2)\!=\! \iint u({z_1+z_2})\, P(t,x_0,dz_1)\, \widetilde{P}(t,x_0,dz_2).$$ This immediately yields the $C_b$-Feller property of $(T_t^S)_{t\ge0}$.

On the other hand, let $({{A}_w^S}, D({{A}_w^S}))$ be the weak infinitesimal operator of the process $({X}^S_t)_{t\ge0}$.
Again from the proof of Theorem \ref{th21th1} (iii) we deduce that $C_c^\infty(\R^d)\subset D({{A}_w^S})$, and ${{A}_w^S}|_{C_c^\infty(\R^d)}=-p^S(\cdot,D)$ is a pseudo differential operator with symbol $p^S(x,\xi)=2\Re p(x,\xi/2)$. Indeed, for any $u\in C_c^\infty(\R^d)$ and $x\in\R^d$, using \eqref{symmet1}, the Fourier transform and the Fubini theorem, we get
\begin{align*}
    &\lim_{t\to 0} \frac{T_t^Su(x)-u(x)}{t}\\
    &=\lim_{t\to 0} \frac{1}{t}\bigg( \iint u\Big( \frac{z_1+z_2}{2}\Big)\,P(t,x,dz_1)\,\widetilde{P}(t,x,dz_2)-u(x)\bigg)\\
    &=\lim_{t\to0}\frac{1}{t}\bigg( 2^d \iint e^{i\langle \xi, x+z_2\rangle}\,\lambda_t(x,\xi)\, \widehat{u}(2\xi)\,d\xi\,\widetilde{P}(t,x,dz_2)-2^d\int e^{i\langle 2\xi, x\rangle} \,\widehat{u}(2\xi)\,d\xi \bigg)\\
    &=\lim_{t\to0}\frac{1}{t}\bigg( 2^d \int e^{i\langle \xi, 2x\rangle}\,\lambda_t(x,\xi)\, \widehat{u}(2\xi)\,d\xi\!\int e^{i\langle \xi, z_2-x\rangle} \widetilde{P}(t,x,dz_2)-2^d\int e^{i\langle 2\xi, x\rangle} \,\widehat{u}(2\xi)\,d\xi \bigg)\\
     &=\lim_{t\to0}\frac{1}{t}\bigg( 2^d \int e^{i\langle \xi, 2x\rangle}\,\lambda_t(x,\xi)\lambda_t(x,-\xi)\, \widehat{u}(2\xi)\,d\xi-2^d\int e^{i\langle 2\xi, x\rangle} \,\widehat{u}(2\xi)\,d\xi \bigg)\\
    &=\lim_{t\to0}2^d \bigg( \int e^{i\langle 2\xi, x\rangle}\,\frac{|\lambda_t(x,\xi)|^2-1}{t}\, \widehat{u}(2\xi)\,d\xi\bigg) \\
         &=-2^d\bigg(\int e^{i\langle 2\xi, x\rangle}\,\big( p(x,\xi)+ p(x,-\xi)\big)\, \widehat{u}(2\xi)\,d\xi\bigg)\\
    &=-\int e^{i\langle \xi, x\rangle}\,\big(2\!\Re p(x,\xi/2)\big)\, \widehat{u}(\xi)\,d\xi\\
    &=-p^S(x,D)u(x).
\end{align*}
The second and the forth equalities follow from the fact that the characteristic functions corresponding to ${P}(t,x,dy)$ and $\widetilde{P}(t,x,dy)$ are $\lambda_t(x,\xi)$ and $\widetilde{\lambda}_t(x,\xi)=\lambda_t(x,-\xi)$, respectively.  In the third equality from below we used Theorem \ref{th21th1} (i) and the dominated convergence theorem. Therefore, the weak infinitesimal operator of $({T}^S_t)_{t\ge0}$ on $C_c^\infty(\R^d)$ is just $-p^S(\cdot,D)$.
According to \cite[Chapter I, (1.49), Page 40]{Dy}, $(({X}^S_t)_{t\ge0},({\Pp}^x)_{x\in\R^d})$ is the solution to the martingale problem for $(-p^S(\cdot,D),C_c^\infty(\R^d)).$

Furthermore, according to Lemma \ref{th21th11}, we get
\begin{equation*}
    \varlimsup_{r\rightarrow\infty}\sup_{x\in\R^d}\sup_{|\xi|\le 1/r}\Re p(x,\xi/2)=0,
\end{equation*}
Hence, Lemma \ref{th21th11} and Proposition \ref{appendix3} in the appendix finally imply that $({T}^S_t)_{t\ge0}$ is a Feller semigroup, and so $({X}^S_t)_{t\ge0}$ is a Feller process.

Let $(A^S,{D}(A^S))$ be the Feller generator of $({X}^S_t)_{t\ge0}$. According to \cite[Lemma 31.7, Page 209]{SA} and the conclusion above, $C_c^\infty(\R^d)\subset D(A^S)$ and $A^S|_{C_c^\infty(\R^d)}=-p^S(\cdot,D)$. Again by the independence of $(X_t)_{t>0}$ and $(\widetilde{X})_{t>0}$ and the fact that $X_0 \sim \widetilde{X}_0$ we see that for any $t>0$ the characteristic function of $({X}^S_t)_{t\ge0}$ is given by
\begin{align*}
    \lambda^S_t(x,\xi)
    &=\Ee^x\big(e^{i\langle X^S_t-x,\xi\rangle}\big)\\
    &=\Ee^x\big(e^{i\langle (X_t-x)/2,\xi\rangle}\times e^{i\langle (\widetilde{X}_t-x)/2,\xi\rangle}\big)\\
        &=\Ee^x\big(e^{i\langle X_t-x,\xi/2\rangle}\times e^{i\langle \widetilde{X}_t-x,\xi/2\rangle}\big)\\
    &=\Ee^{x}\big(e^{i\langle X_t-x,\xi/2\rangle}\big)\Ee^{x}\big( e^{i\langle \widetilde{X}_t-x,\xi/2\rangle}\big)\\
    &=|\lambda_t(x,\xi/2)|^2.
\end{align*}
Together with Theorem \ref{th21th1} (i) this also shows that the symbol of the process $({X}^S_t)_{t\ge0}$ is $2\Re p(x,\xi/2)$.
\end{proof}

\section{Proof of Theorem \ref{pro1} and some Applications}\label{section3}

\subsection{Proof of Theorem \ref{pro1}}

\begin{proof}
\textbf{(i)} For any $t>0$,
\begin{align*}
    \|T_t\|_{1\to \infty}
    &:= \sup_{u\in L^1(dx),\,\|u\|_1=1}\|T_tu\|_\infty\\
    &=\sup_{u\in C_c^\infty(\R^d),\,\|u\|_1=1}\|T_tu\|_\infty\\
    &=\sup_{u\in C_c^\infty(\R^d),\,\|u\|_1=1}\;\sup_{x\in\R^d}\bigg|\int e^{i\langle x,\xi\rangle}\hat{u}(\xi)\lambda_t(x,\xi)\,d\xi\bigg|\\
    &\le \sup_{u\in C_c^\infty(\R^d),\,\|u\|_1=1}\;\sup_{x\in\R^d}\int |\hat{u}(\xi)||\lambda_t(x,\xi)|\,d\xi\\
    &\le (2\pi)^{-d}\sup_{x\in\R^d}\int\big|\lambda_t(x,\xi)\big|\,d\xi,
\end{align*}
where we have used that for all $\xi\in\R^d$, $|\hat{u}(\xi)|\le(2\pi)^{-d}\|u\|_1$. By assumption \eqref{pro11} and Theorem \ref{th2},
$$
    \|T_t\|_{1\to \infty}
    \le (2\pi)^{-d}\int \exp\Bigl({-\frac{t}{16}\inf_{z\in\R^d} \Re p(z,2\xi)}\Bigr)\,d\xi<\infty,
$$
which yields the ultracontractivity of the Feller semigroup. Now we can get the existence of the transition density and the strong Feller property of the semigroup from \cite[Proposition 3.3.11]{Wang} and \cite[Corollary 2.2]{SWSFP}, respectively.

\noindent\textbf{(ii)}
The assertion follows essentially from \cite[Theorem 2.2]{WJ} and Theorem \ref{th2}. For the readers' convenience we
repeat the relevant part of the argument from \cite[Theorem 2.2]{WJ}. For $x=(x_1,\ldots,x_d)$ and $r>0$, write
$$
    Q(x,r)
    :=\Big\{z=(z_1,\ldots,z_d)\in\R^d\::\: |z_j-x_j|\le r \text{\ for\ } 1\le j\le d\Big\}.
$$
For any $x\in\R^d$ and $r>0$, define
$$
    g(y)
    =g_x(y):=
    \begin{cases}
        \displaystyle r^{2d}, & \text{if\ \ } y=x,\\[\bigskipamount]
        \displaystyle \prod_{j=1}^d\Bigl(\frac{\sin r (y_j-x_j)}{y_j-x_j}\Bigr)^2, & \text{if\ \ } y\neq x.
    \end{cases}
$$
Then, $g\in B_b(\R^d)\cap L^1(\R^d)$ and
$$
    \hat{g}(\xi)
    =(8\pi)^{-d}e^{-i\langle x,\xi\rangle}\biggl(\bigotimes_{j=1}^d{\I}_{[-r, r]}*\bigotimes_{j=1}^d{\I}_{[-r, r]}\biggr)(\xi),
$$
cf.\ \cite[Table 3.5.19, Page 117, Vol.\ 1]{jac-book}. In particular, $\hat{g}\in L^1(\R^d)$. According to the proof of
\cite[Theorem 1.1]{J1}, see also \cite[Remark (B), Page 65]{J1}, \eqref{char2} holds for the test function $g$. That is,
$$
    T_tg(x)=\int e^{i\langle x,\xi\rangle}\hat{g}(\xi)\lambda_t(x,\xi)\,d\xi.
$$

Since $\big|\frac{\sin r}{r}\big|\ge {1}/{2}$ for $|r|\le \pi/3$, we know that $g(y)\ge (4^{-1}r)^{2d}$ for all $y\in Q(x,{\pi}/{(3r)})$. For $s>0$, write $X_s=(X_s^1,\ldots,X_s^d)$. By monotone convergence,
\begin{align*}
    (4^{-1}r)^{2d}\, &\Ee^{x}\biggl(\int_0^\infty \bigotimes_{j=1}^d{\I}_{[x_j-\pi/(3r),\, x_j+\pi/(3r)]}(X_s^j)\,ds\biggr)\\
    &= \lim_{\alpha\rightarrow0}\Ee^{x}\biggl(\int_0^\infty e^{-\alpha t} (4^{-1}r)^{2d} \bigotimes_{j=1}^d{\I}_{[x_j-\pi/(3r),\, x_j+\pi/(3r)]}(X^j_s)\,ds\biggr)\\
    &\le \lim_{\alpha\rightarrow0} \int_0^\infty e^{-\alpha t}T_t g(x)\,dt\\
    &= \lim_{\alpha\rightarrow0}\int_0^\infty e^{-\alpha t}\,dt\int e^{i\langle x, \xi\rangle}{\hat{g}(\xi)}\lambda_t(x,\xi)\,d\xi\\
    &= (8\pi)^{-d}\lim_{\alpha\rightarrow0}\int_0^\infty e^{-\alpha t}\int \biggl(\bigotimes_{j=1}^d{\I}_{[-r, r]}*\bigotimes_{j=1}^d{\I}_{[-r, r]}\biggr)(\xi)\,\lambda_t(x,\xi)\,d\xi\,dt\\
    &= (8\pi)^{-d}\lim_{\alpha\rightarrow0}\int_0^\infty e^{-\alpha t}\int \biggl(\bigotimes_{j=1}^d{\I}_{[-r, r]}*\bigotimes_{j=1}^d{\I}_{[-r, r]}\biggr)(\xi)\,\Re\lambda_t(x,\xi)\,d\xi\,dt\\
    &\le(8\pi)^{-d}\lim_{\alpha\rightarrow0}\int_0^\infty e^{-\alpha t}\int \biggl(\bigotimes_{j=1}^d{\I}_{[-r, r]}*\bigotimes_{j=1}^d{\I}_{[-r, r]}\biggr)(\xi)\,|\Re\lambda_t(x,\xi)|\,d\xi\,dt.
\end{align*}
In the penultimate line we have used that the function $\bigotimes_{j=1}^d{\I}_{[-r, r]}*\bigotimes_{j=1}^d{\I}_{[-r, r]}$ is symmetric. Note that for all $\xi\in\R^d$
$$
    \biggl(\bigotimes_{j=1}^d{\I}_{[-r,r]}*\bigotimes_{j=1}^d{\I}_{[-r, r]}\biggr)(\xi)
    \le (2r)^d\biggl(\bigotimes_{j=1}^d{\I}_{[-2r, 2r]}\biggr)(\xi)
    \leq (2r)^d \I_{Q(0,2r)}(\xi).
$$
This inequality and \eqref{th21} give
\begin{align*}
    \Bigl(\frac{\pi r}{4}\Bigr)^{d} &\Ee^{x}\biggl(\int_0^\infty \bigotimes_{j=1}^d{\I}_{[x_j-\pi/(3r),\, x_j+\pi/(3r)]}(X_s^j)\,ds\biggr)\\
    &\le \lim_{\alpha\rightarrow0}\int_0^\infty e^{-\alpha t} \int_{Q(0,2r)} |\Re\lambda_t(x,\xi)|\,d\xi\,dt\\
    &\le\int_0^\infty\int_{Q(0,2r)}|\Re\lambda_t(x,\xi)|\,d\xi  \, dt\\
    &\le \int_0^\infty \int_{\{|\xi|\le 2r\sqrt{d}\}} |\lambda_t(x,\xi)|\,d\xi\, dt\\
    &\le \int_0^\infty \int_{\{|\xi|\le 2r\sqrt{d}\}} \exp \biggl(-{\frac{t}{16}}\inf_{z\in\R^d}\Re p(z,2\xi)\biggr)\,d\xi\, dt \\
    &=16\int_{\{|\xi|\le 2r\sqrt{d}\}}\frac{d\xi}{\inf_{z\in\R^d}\Re p(z,2\xi)}\,.
\end{align*}
Therefore, for any $r>0$,
$$
    \Ee^{x}\Bigl(\int_0^\infty \I_{Q(x,\pi/(3r))}(X_s)\,ds\Bigr)
    \le\frac{4^{d+2}}{(\pi r)^d}\int_{\{|\xi|\le 2r\sqrt{d}\}}\frac{d\xi}{\inf_{z\in\R^d}\Re p(z,2\xi)}\,.
$$
Since $r>0$ is arbitrary, the assertion follows because of \eqref{pro12}.

 \smallskip\noindent
    \textup{\bfseries (iii)}
Our proof follows Berman's argument, see \cite[Chapter V, Theorem 1.1 (1), Page 126]{Ber} and \cite[Section 3]{Berm}. The occupation measure $\mu_t$ of the time interval $[0,t]$, $t>0$, is defined through the relation
$$
    \int_{\R^d} f(x)\,\mu_t(dx)=\int_0^t f(X_s)\,ds
    \qquad\text{for all}\quad f\in B_b(\R^d), \; f\geq 0.
$$
Define the measure $\mu$ by
$$
    \mu(dx) := \int_0^\infty e^{-t}\,\mu_t(dx)\,dt
$$
in the vague topology of measures. In particular, each $\mu_t$ is absolutely continuous with respect to $\mu$ with a density bounded from above by $e^t$. We claim that
\begin{equation}\label{local1}
    \int_{\R^d} \Ee^x\big(|\widehat{\mu}(\xi)|^2\big)\,d\xi
    <\infty
    \qquad\text{for every\ \ }x\in\R^d.
\end{equation}
Using Fubini's theorem and then Plancherel's theorem we conclude that, almost surely, $\mu$ has a square-integrable density $\frac{d\mu}{dx}$ with respect to $dx \otimes d\Pp$. By the definition of $\mu_t$ and $\mu$, the local time of the process is just $L(x,t)=e^{t} \times\frac{ d\mu}{dx}$ for all $ x\in\R^d$ and $t\ge0$, and so the required assertion follows.

All that remains to be done is to establish \eqref{local1}. From the very definition of $\mu$, one has
\begin{align*}
    &\Ee^x\big(|\widehat{\mu}(\xi)|^2\big) \;=\; \Ee^x\big(\widehat{\mu}(\xi)\widehat{\mu}(-\xi)\big)\\
    &=\Ee^x\biggl[\biggl(\int_0^\infty e^{-s} e^{i\langle X_s,\xi\rangle}\,ds\biggr)\biggl(\int_0^\infty e^{-t} e^{-i\langle X_t,\xi\rangle}\,dt\biggr)\biggr]\\
    &=\Ee^x\biggl(\int_0^\infty \int_0^\infty e^{-(s+t)} e^{i \langle X_s-X_t,\xi\rangle}\,ds\,dt\biggr)\\
    &=\Ee^x\biggl(\int_0^\infty \int_t^\infty e^{-(s+t)} e^{i \langle X_s-X_t,\xi\rangle}\,ds\,dt\biggr)
        + \Ee^x\biggl(\int_0^\infty \int_0^t e^{-(s+t)} e^{i \langle X_s-X_t,\xi\rangle}\,ds\,dt\biggr)\\
    &=\Ee^x\biggl(\int_0^\infty \int_t^\infty e^{-(s+t)} e^{i \langle X_s-X_t,\xi\rangle}\,ds\,dt\biggr)
        + \Ee^x\biggl(\int_0^\infty \int_s^\infty e^{-(s+t)} e^{i \langle X_s-X_t,\xi\rangle}\,dt\,ds\biggr)\\
    &=2 \Ee^x\biggl(\int_0^\infty \int_t^\infty e^{-(s+t)} \Re e^{i \langle X_s-X_t,\xi\rangle}\,ds\,dt\biggr)\\
    &=2 \int_0^\infty \int_t^\infty e^{-(s+t)}\Re \Ee^x\Bigl( e^{i \langle X_s-X_t,\xi\rangle}\Bigr)\,ds\,dt\\
    &=2 \int_0^\infty \int_t^\infty e^{-(s+t)} \Ee^x \left(\Re \Ee^{y}\Bigl( e^{i \langle X_{s-t}-y,\xi\rangle}\Bigr) \Big|_{y=X_t}\right)\,ds\,dt.
\end{align*}
In the last step we have used the Markov property. From \eqref{th21} we conclude that
\begin{align*}
    \Ee^x\big(|\widehat{\mu}(\xi)|^2\big)
    &\le 2 \int_0^\infty \int_t^\infty e^{-(s+t)}\sup_{z\in\R^d}\Big|\Re \lambda_{s-t} (z,\xi)\Big|\,ds\,dt\\
     &\le 2 \int_0^\infty \int_t^\infty e^{-(s+t)}\sup_{z\in\R^d}\Big|\lambda_{s-t} (z,\xi)\Big|\,ds\,dt\\
    &\le 2\int_0^\infty \int_t^\infty e^{-(s+t)- \frac1{16}(s-t)\inf_{z\in\R^d}\Re p(z,\xi)}\,ds\,dt\\
    &=2\int_0^\infty \int_t^\infty e^{-(s-t)-\frac1{16}(s-t)\inf_{z\in\R^d}\Re p(z,\xi)}\,ds\, e^{-2t}\,dt\\
    &=\frac{16}{16+\inf_{z\in\R^d}\Re p(z,\xi)}\,.
\end{align*}
This estimate and the assumption \eqref{local} show that \eqref{local1} holds.
\end{proof}

\begin{remark}\label{transientremark}
\textup{\bfseries (i)} A close inspection of the proofs of Theorem \ref{pro1} (ii) and (iii) shows that the transience and the existence of local times for a Feller process only depend on $\Re \lambda_t(x,\xi)$, i.e.\ the real part of the characteristic function. This is familiar from the theory of L\'{e}vy processes.

\smallskip\noindent
\textup{\bfseries (ii)}
If for every $x\in\R^d$ the symbol $\Re p(x,\xi)$ is a function of $|\xi|$ which is unbounded in $\xi$, i.e.\ if for every $x\in\R^d$, $\Re p(x,\xi)=\Re p(x,|\xi|)$ and
$$
    \varliminf_{|\xi|\to \infty}\Re p(x,|\xi|)=\infty,
$$
then we can replace the condition `for every $r>0$' in \eqref{pro12} by `for some $r>0$'.

This can be seen from the following argument: first,
$$
    \int_{\{|\xi|\le r\}}\frac{d\xi}{\inf_{z\in\R^d}\Re p(z,\xi)} =\infty
$$
is equivalent to saying that$$
    \int_{\{|\xi|\le r\}}\frac{d\xi}{\sup_{z\in\R^d}\big(-\Re p(z,\xi))} =-\infty.
$$
Now, if there exists $r_0>0$ such that
$$
    \int_{\{|\xi|\le r_0\}}\frac{d\xi}{\inf_{z\in\R^d} \Re p(z,\xi)}<\infty
$$
and
$$
    \int_{\{|\xi|\le r\}}\frac{d\xi}{\inf_{z\in\R^d} \Re p(z,\xi)} =\infty
    \qquad\text{for all\ \ }r>r_0.
$$
Then,
$$
    \int_{\{|\xi|\le r_0\}}\frac{d\xi}{\sup_{z\in\R^d} \big(- \Re p(z,\xi)\big)} <-\infty
$$
and
$$
    \int_{\{|\xi|\le r\}}\frac{d\xi}{\sup_{z\in\R^d}  \big(- \Re p(z,\xi)\big)} =-\infty
    \qquad\text{for all\ \ }r>r_0.
$$
Hence,
$$
    \int_{\{r_0\le |\xi|\le r\}}\frac{d\xi}{\sup_{z\in\R^d} \big(- \Re p(z,\xi)\big)} =-\infty
    \quad\text{for all\ \ } r>r_0.
$$
Thus, there exists a sequence $(\xi_n)_{n\ge1}\subset\{z\in\R^d: r_0\le |z|\le r\}$ such that $$\lim_{n\to \infty} \sup_{z\in\R^d} \Re p(z,\xi_n)=0.$$ In particular, for all $x\in\R^d$, $\lim\limits_{n\to \infty} \Re p(x,\xi_n)=0$.

By compactness, there is a subsequence $(\xi'_n)_{n\ge1}$ of $(\xi_n)_{n\ge1}$ such that $\lim_{n\to \infty}\xi'_n=\xi_0$. Since the function $\xi\mapsto \Re p(x,\xi)$ is continuous for any fixed $x\in\R^d$, we get that
$ \Re p(x,\xi_0)=0$ for every $x\in\R^d$. Thus, for any $x\in\R^d$ and $\eta\in \R^d$ with $|\eta|=|\xi_0|$, $\Re p(x,\eta)=0$. Since $\sqrt{\Re p(x,\cdot)}$ is subadditive, $\xi\mapsto \Re p(x,\xi)$ is periodic. Because of \eqref{assumption}, there is a constant $C=C(\xi_0)$ such that
$$
    \sup_{x,\xi\in\R^d}\Re p(x,\xi)\le C
$$
which cannot be the case since $\Re p(x,\xi)$ is unbounded.
\end{remark}

\subsection{Examples}
\subsubsection*{{\bfseries Feller Processes with Real Symbol Obtained by Variable Order Subordination}}
Let $\psi$ be a real-valued negative definite function on $\R^d$ such that $\psi(0)=0$. Let $f:\R^d\times[0,\infty)\to [0,\infty)$ be a measurable function such that $\sup_x f(x,s)\le c (1+s)$ for some constant $c>0$, and for fixed $x\in\R^d$ the function $s\mapsto f(x,s)$ is a Bernstein function with $f(x,0)=0$. Bernstein functions are the characteristic Laplace exponents of subordinators; our standard reference is the monograph \cite{ssv}. Then,
$$
    q(x,\xi):=f(x,\psi(\xi))
$$
is a real-valued symbol satisfying \eqref{assumption}. Since $f(x,s)=s^{r(x)}$ where $r:\R^d\to[0,1]$ is a possible choice for $f$, this class includes symbols describing variable (fractional) order of differentiation or variable order fractional powers. We refer to \cite{EJ} and the references therein for more details on Feller semigroups obtained by variable order subordination. According to Theorem \ref{pro1} and Remark \ref{transientremark}, we see

\begin{corollary}
Let $(X_t)_{t\geq 0}$ be a Feller process with the symbol $q(x,\xi) = f(x,\psi(\xi))$ above. Set $f_0(s):=\inf_{x\in\R^d} f(x,s)$ for $s\in[0,\infty)$. Then, we have

\smallskip\noindent
\textup{\bfseries (i)}
If
$$
    \lim_{|\xi|\to \infty}\frac{ f_0(\psi(\xi))}{\log (1+|\xi|)}=\infty,
$$
then the corresponding Feller semigroup $(T_t)_{t\ge0}$ is ultracontractive and has the strong Feller property.

\smallskip\noindent
\textup{\bfseries (ii)}
If
$$
    \int_{\{|\xi|\le r\}}\frac{d\xi}{f_0(\psi(\xi))} <\infty
    \qquad\text{for every\ \ }r>0,
$$
then the Feller process $(X_t)_{t\ge0}$ is transient.

\smallskip\noindent
\textup{\bfseries (iii)}
If
$$
    \int\frac{d\xi}{1+f_0(\psi(\xi))} <\infty,
$$
then the Feller process $(X_t)_{t\ge0}$ has local times.

\smallskip

If the symbol $\psi(\xi)$ only depends on $|\xi|$, i.e.\ if $\psi(\xi) = \phi(|\xi|)$ for some function $\phi$, then it is enough to assume that the condition in \textup{(ii)} holds for some $r>0$.
\end{corollary}

\subsubsection*{{\bfseries Rich Bass' Stable-like Processes}}
A stable-like process on $\R^d$ is a Feller process, whose generator
has the same form as that of a rotationally symmetric stable
L\'{e}vy motion, but the index of `stability' depends on the state space,
see \cite{BAss}. The infinitesimal generator is of the form
$$
    L^{(\alpha)}u(x)
    =\int_{z\neq 0} \Bigl(u(x+z)-u(x)-\langle\nabla u(x),z\rangle\I_{\{|z|\le1\}}\Bigr)\,\frac{C_{\alpha(x)}}{|z|^{d+\alpha(x)}}\,dz,
    \qquad u\in C_b^2(\R^d),
$$
where $0<\alpha(x)<2$ and $C_{\alpha(x)}$ is a constant defined through the L\'{e}vy-Khintchine formula
$$
    |\xi|^{\alpha(x)}
    =C_{\alpha(x)}\int_{z\neq 0} \Bigl(1-\cos \langle \xi, z\rangle\Bigr)\,\frac{dz}{|z|^{d+\alpha(x)}},
$$
i.e.\
$$
    C_{\alpha(x)}
    =\alpha(x)2^{\alpha(x)-1}\Gamma\big((\alpha(x)+d)/2\big)\Big/\Bigl(\pi^{d/2}\Gamma\big(1-\alpha(x)/2\big)\Bigr),
$$
see \cite[Exercise 18.23, Page 184]{BF}. In other words, the operator $L^{(\alpha)}$ can be regarded as a pseudo differential operator of variable order with symbol $|\xi|^{\alpha(x)}$, i.e.\ $L^{(\alpha)}=-(-\Delta)^{\alpha(x)/2}$.

\begin{theorem}\label{strong}
Assume that $\alpha(x)\in C_b^1(\R^d)$ such that $0<\underline{\alpha}=\inf \alpha \le \alpha(x)\le \sup \alpha=\overline{\alpha}<2$. Then, there exists a Feller process $(X_t)_{t\geq 0}$ (which we call stable-like process in the sense of R.\ Bass) having the symbol $|\xi|^{\alpha(x)}$, such that the following statements hold.

\smallskip\noindent
\textup{\bfseries (i)}
The Feller semigroup $(T_t)_{t\ge0}$ of $(X_t)_{t\ge0}$ has the strong Feller property, and the transition probability $P(t,x,dy)$ of $(X_t)_{t\ge0}$ has a density function $p(t,x,y)$ with respect to Lebesgue measure; moreover
$$
    \sup_{x,y\in\R^d} p(t,x,y)\le \begin{cases}
        C \, t^{-d/\underline{\alpha}}
        &\text{for small\ \ }t\ll 1;\\[\bigskipamount]
        C \, t^{-d/\overline{\alpha}}
        &\text{for large\ \ }t\gg 1.
    \end{cases}
$$

\smallskip\noindent
\textup{\bfseries (ii)}
If $d\ge 2$, then the process $(X_t)_{t\ge0}$ is transient.

\smallskip\noindent
\textup{\bfseries (iii)}
If $d=1$ and $\sup_{|x|\ge K}\alpha(x)\in (0,1)$ for some constant $K>0$, then the process $(X_t)_{t\ge0}$ is transient.

\smallskip\noindent
\textup{\bfseries (iv)}
If $d=1$ and $\inf_{x\in\R}\alpha(x)\in (1,2)$, the process $(X_t)_{t\ge0}$ has local times.
 \end{theorem}

Before we begin with the proof of Theorem \ref{strong}, a few words on related work on stable-like processes is appropriate.
\begin{remark}
\textup{\bfseries (i)}
Under the condition that $\alpha(\cdot)\in C_b^\infty(\R^d)$, the strong Feller property of stable-like processes has been established in \cite[Theorem 3.3]{SWSFP}. In addition to this, our result provides an upper bound for on-diagonal estimates of the heat kernel of stable-like processes. Note that a stable-like process is \emph{not} symmetric, i.e.\ Dirichlet form methods fail if we want to derive estimates as in Theorem \ref{strong} (i).

\smallskip
\noindent\textup{\bfseries (ii)}
If $\alpha (x)$ is Dini continuous and $\inf_{x\in\R}\alpha(x)\in (1,2)$, the existence of local times for stable-like processes was shown by Bass \cite[Theorem 2.1]{BAss1}. Bass' technique is different from ours.

\smallskip
\noindent\textup{\bfseries (iii)}
Recurrence and transience of a particular class of one-dimensional stable-like processes (with discontinuous exponents) have been studied in \cite{Bjorn} using an overshoot approach under the assumption that the underlying process is a \emph{Lebesgue-irreducible T-process}. Although the setting in \cite[Corollary 5.5]{Bjorn} is different from the situation here, we remark that our proof shows that a stable-like process with $\alpha(x)\in C_b^1(\R^d)$ and $0<\underline{\alpha}=\inf \alpha \le \alpha(x)\le \sup \alpha=\overline{\alpha}<2$ is a \emph{Lebesgue-irreducible T-process}.
\end{remark}

\begin{proof}[Proof of Theorem \ref{strong}]
According to \cite[Corollary 2.3]{BAss} the solution to the martingale problem for $(L^{(\alpha)}, C_0^\infty(\R^d))$ is well posed. Therefore there exists a unique strong Markov process $((X_t)_{t\ge0},(\Pp^x)_{x\in\R^d})$ for which $\Pp^x$ solves the martingale problem for $(L^{(\alpha)}, C_0^\infty(\R^d))$ at each point $x\in\R^d$. For any $t\ge0$, $x\in\R^d$ and $f\in B_b(\R^d)$, we define
$$
    T_tf(x)=\Ee^x(f(X_t)).
$$
From \cite[Propositions 6.1 and 6.2]{BAss}, we know that $(T_t)_{t\ge0}$ is a Markov semigroup which has the $C_b$-Feller property; that is, for any $t\ge0$, $T_t$ maps the set of bounded continuous functions into itself. By Proposition \ref{appendix3} we see that $T_t$ enjoys the Feller property, i.e.\ $T_t$ maps the set of continuous functions vanishing at infinity into itself. Note that the uniqueness of the solution for the martingale problem indicates that $C_c^\infty(\R^d)$ is contained in the extended domain of the operator $L^{(\alpha)}$. On the other hand, it is easy to check that, under our assumptions on the index function $\alpha$, we have $L^{(\alpha)}u\in C_\infty(\R^d)$
for any $u\in C_c^\infty(\R^d)$. Thus, Proposition \ref{prop44} shows that $C_c^\infty(\R^d)$ actually is contained in the domain of the operator $L^{(\alpha)}$.
Therefore, (i), (ii) and (iv) follow from Theorem \ref{pro1}.

To prove the assertion (iii) we need a few auxiliary results on stable-like processes. Let $P(t,x,dy)$ be the transition function of $(X_t)_{t\ge0}$ and denote its density by $p(t,x,y)$. Note that $x\mapsto C_{\alpha(x)}$ is a positive function of class $C_b^\infty(\R^d)$. From \cite[Theorem 5.1 and its Corollary, Pages 759--760]{Kolo} we know that $p(t,x,y)$ is strictly positive everywhere on $(0,\infty)\times \R^d\times \R^d$. Therefore, $(X_t)_{t\ge0}$ is Lebesgue irreducible, i.e.\ for any Borel measurable set $A$ with $\mathrm{Leb}(A)>0$ and $x\in\R^d$, $\Ee^x\big(\int_0^\infty \I_{A}(X_t)\,dt\big)>0$. Recall that, a Markov process $(X_t)_{t\ge0}$ is Harris recurrent, if for any Borel measurable set $A$ with $\mathrm{Leb}(A)>0$ and $x\in\R^d$, $\Ee^x\big(\int_0^\infty \I_{A}(X_t)\,dt\big)=\infty$. The Lebesgue irreducibility and the strong Feller property yield that the stable-like process $(X_t)_{t\ge0}$ is either Harris recurrent or transient, see e.g.\ \cite[Theorem 3.2 (a)]{MT} and \cite[Theorem 2.3]{Tw}.
Moreover, we know from \cite[Theorem 3.3]{MTT} that $(X_t)_{t\ge0}$
\begin{itemize}
\item[] is Harris recurrent if, and only if,  $\Pp^x(\sigma_{_{\overline{B(0,R)}}}<\infty)=1$ for every $x\in\R^d$;
\item[] is transient if, and only if, $\Pp^x(\sigma_{_{\overline{B(0,R)}}}<\infty)<1$ for some $x\in\R^d$,
\end{itemize}
where $\sigma_{_{\overline{B(0,R)}}}$ is the first entrance time of the process into $\overline{B(0,R)}$ and $R>0$ is any fixed radius. From these characterizations, we conclude that any two stable-like processes which coincide outside some compact set have the same (Harris) recurrence and transience behaviour, see \cite[Theorems 4.6 and 4.7]{Bjorn}.

Now we can use Theorem \ref{pro1} (ii) to infer that a one-dimensional stable-like process is transient, if $\sup_{x\in\R}\alpha(x)\in (0,1)$.
Therefore, (iii) follows from this conclusion and the remark above.
\end{proof}

\section{appendix}
Let $(X_t)_{t\ge0}$ be a Feller process with generator $(A,D(A))$ and semigroup $(T_t)_{t\ge0}$.
Let us first comment on the assumption that
\begin{equation}\label{sss}
    \emph{the test functions $C_c^\infty(\R^d)$ are contained in the domain $D(A)$}
\end{equation}
of the Feller generator $A$. Usually \eqref{sss} is not easy to verify in applications; on the other hand, we do not know many non-trivial examples of Feller processes which do not satisfy \eqref{sss}. In what follows, we will make full use of the extended domain of the Feller generator $A$, which is easier to deal with than the domain $D(A)$.

Recall that for a strong Markov process $(X_t)_{t\ge0}$ on $\R^d$ with infinitesimal generator $(A,D(A))$, the \emph{extended domain} $\widetilde{{D}}(A)$ is defined by
\begin{align*}
    \widetilde{D}(A)=&\Bigg\{u\in B(\R^d)\::\: \text{\ there is a measurable function\ \ } g \text{\ \ such that}\\
    &\qquad \biggl(u(X_t)-\int_0^t g(X_s)\,ds,\,\,\mathscr{F}_t\biggr)_{t\geq 0}\text{\ is a local martingale under\ } \Pp^x\Bigg\},
\end{align*}
where $\mathscr{F}_t := \sigma(X_s:s\le t)$ is the natural filtration of the process $(X_t)_{t\geq 0}$, and $B(\R^d)$ is the space of Borel measurable functions on $\R^d$. The function $g$ appearing in the definition of $\widetilde D(A)$ need not be unique, cf.\ \cite[Chapter 1, Page 24]{EK}. If, however, $Au$ can be defined, $g=Au$ is admissible; in particular, $D(A)\subset\widetilde D(A)$. Conversely, we can use the situation where $g$ is unique to extend the operator $(A,D(A))$. The concept of extended domain is similar to the full generator for a contraction semigroup in \cite[Chapter 1, Pages 23--24]{EK}. 

For a Feller generator $(A,D(A))$ such that $C_c^\infty(\R^d)  \subset D(A)$ one has $C_c^\infty(\R^d)\subset \widetilde{{D}}(A)$, see \cite[Chapter 4, Proposition 1.7]{EK} and \cite[Lemma 2.3 and Corollary 3.6]{S3}; on the other hand, the condition $C_c^\infty(\R^d) \subset D(A)$ along with the assumption \eqref{assumption} implies that $C_c^\infty(\R^d)\subset C_b^2(\R^d)\subset \widetilde{{D}}(A)$, see Proposition \ref{appendix1} below for the simple proof of the assertion that $C_b^\infty(\R^d)\subset \widetilde{{D}}(A)$, where $C_b^\infty(\R^d)$ is the space of arbitrarily often differentiable functions such that the function and its derivatives are bounded. Conversely, we have

\begin{proposition}\label{prop44}
    Let $(X_t)_{t\ge0}$ be a Feller process with generator $(A,D(A))$. Suppose that $C_c^\infty(\R^d)\subset\widetilde{D}(A)$, and that for any $u\in C_c^\infty(\R^d)$ there is (an extension of $A$) such that $Au$ is well-defined and in $C_\infty(\R^d)$, the space of continuous functions vanishing at infinity. If the process $(X_t)_{t\ge0}$ is conservative, then $C_c^\infty(\R^d)\subset D(A)$.
    \end{proposition}
    \begin{proof}
    The Feller semigroup $(T_t)_{t\ge0}$ has a unique extension on $B_b(\R^d)$ (the space of the bounded Borel measurable functions), cf.\ \cite[Section 3]{S1}. For simplicity, we still denote by $(T_t)_{t\ge0}$ this extension. Since the process $(X_t)_{t\ge0}$ is conservative, $T_t1=1$ for every $t\ge0$. According to \cite[Corollary 3.4]{S1}, $t\mapsto T_tu$ is for all $u\in C_b(\R^d)$ continuous with respect to locally uniform convergence.

    Let $\tau_{_{{B(x,r)}}}$ be the first exit time of the process from the open ball ${B(x,r)}$. Since $C_c^\infty(\R^d)\subset\widetilde{D}(A)$, for any $x\in\R^d$, $r>0$ and $u\in C_c^\infty(\R^d)$,
$$
    \Ee^x\biggl(u(X_{t\wedge \tau_{_{B(x,r)}}})-\int_0^{t\wedge \tau_{_{B(x,r)}}}Au(X_{s})\,ds\biggr)
    =u(x).
$$
Since $(X_t)_{t\geq 0}$ is conservative, $\tau_{_{B(x,r)}}\xrightarrow{r\to\infty} \infty$. Thus, we can use the dominated convergence theorem to find that for all $u\in C_c^\infty(\R^d)$
$$
    \Ee^x\biggl(u(X_t)-\int_0^tAu(X_{s})\,ds\biggr)=u(x).
$$

Pick $x\in\R^d$; by the continuity of $t\mapsto T_t(Au)(x)$,
\begin{align*}
    \lim_{t\to 0} \frac{{\Ee}^x\Bigl(u\big({X}_{_{t}}\big)-u(x)\Bigr)}{t}
    =\lim_{t\to 0}\frac{1}{t}\bigg(\int_0^tT_s(Au)(x)\,ds\bigg)
    =Au(x).
\end{align*}
Thus, $u$ belongs to the domain of the weak infinitesimal generator of the process $X_t$. The required assertion follows from \cite[Lemma 31.7, Page 209]{SA}.
\end{proof}

Next, we present a consequence of the assumption $C_c^\infty(\R^d)\subset D(A)$ for Feller processes.

\begin{proposition}\label{appendix1}
Let $(X_t)_{t\ge0}$ be a Feller process with generator $(A,D(A))$ and semigroup $(T_t)_{t\ge0}$. Assume that $C_c^\infty(\R^d)\subset D(A)$ so that $A|_{C_c^\infty(\R^d)}$ is a pseudo differential operator $-p(\cdot,D)$ with symbol $p(x,\xi)$. If \eqref{assumption} is satisfied, then
\begin{equation*}
    T_te_{\xi}(x)=e_{\xi}(x)+\int_0^t T_s A e_{\xi}(x) \,ds
\end{equation*}
holds for all $t>0$ and $x,\xi\in\R^d$, where $e_{\xi}(x)=e^{i\langle \xi, x\rangle}$.
\end{proposition}

\begin{proof}
Denote by $C_b^\infty(\R^d)$ the space of arbitrarily often differentiable functions such that the function and its derivatives are bounded. First we prove that $C_b^\infty(\R^d)$ is contained in the extended domain $\widetilde D(A)$ of the Feller generator $A$.

Let $(b(x),a(x),\nu(x,dz))_{x\in\R^d}$ be the L\'{e}vy
characteristics of the symbol $p(x,\xi)$ given by \eqref{sy}; under the assumption \eqref{assumption}, $c(x)\equiv 0$. Then $A$ has the following representation as an integro-differential operator:
\begin{equation}\label{sf1}\begin{aligned}
    Lf(x)
    &= \frac{1}{2}\sum_{j,k=1}^d a_{jk}(x)\partial_{jk}f(x)+\sum_{j=1}^d b_{j}(x)\partial_j f(x)\\
    &\qquad+\int\limits_{\mathclap{z\neq 0}}\!\!\Bigl(f(x+z)-f(x)-\langle\nabla f(x),z\rangle\I_{\{|z|\le1\}}\Bigr)\,\nu(x,dz).
\end{aligned}\end{equation}
For all $u\in C_c^\infty(\R^d)$ we have $-p(x,D)u(x)=Lu(x)$, $x\in\R^d$, cf.\ \cite[(2.7) and Corollary 2.4]{S3}, and by \cite[Lemma 2.3 and Corollary 3.6]{S3} we have $C_b^2(\R^d)\subset \widetilde{D}(A)$. On the other hand, \cite[Lemma 2.3 and Corollary 3.6]{S3} also show that $(L, D(L))$ is the unique extension of the Feller generator $A$ onto $C_b^2(\R^d)$ such that $\|Lu\|_\infty \le C\|u\|_{C_b^2}$ holds for all $u\in C_b^2(\R^d)$ and some constant $C>0$; here $\|u\|_{C_b^2}:=\sum_{|\alpha|\le 2}\|\partial^\alpha u\|_\infty$.

Let $\chi\in C_c^\infty(\R^d)$ be a smooth cut-off function such that $\I_{B(0,1)}(y)\le \chi(y)\le \I_{B(0,2)}(y)$ for $y\in\R^d$. For $u\in C_b^\infty(\R^d)$ we define $u_n^x(y):=\chi((y-x)/n)u(y)$. Then, $u_n^x\in C_c^\infty(\R^d)$ for every $n\ge1$. By the Taylor formula and the Leibniz rule we see that for any compact set $K\subset\R^d$ there exists a positive constant $C:=C(K,u,n)$ such that $|Lu_n^x(y)|\le C$ for all $y\in K$. Let $\tau_{_{{B(x,r)}}}$ be the first exit time of the process from the open ball ${B(x,r)}$. By the bounded convergence theorem and the fact that $C_c^\infty(\R^d)\subset\widetilde D(L)$, we find for all $x\in\R^d$ and $r,t>0$
\begin{align*}
    \Ee^x\biggl(u(X_{t\wedge \tau_{_{B(x,r)}}})-u(x)\biggr)&=\lim_{n\rightarrow\infty}\Ee^x\biggl(u^x_n(X_{t\wedge \tau_{_{B(x,r)}}})-u^x_n(x)\biggr)\\
    &= \lim_{n\rightarrow\infty}\Ee^x\biggl(\int_0^{t\wedge \tau_{_{B(x,r)}}}Lu^x_n(X_s)\,ds\biggr)\\
    &= \lim_{n\rightarrow\infty}\Ee^x\biggl(\int_{\big(0,\: t\wedge \tau_{_{B(x,r)}}\big)}Lu^x_n(X_{s})\,ds\biggr),
\end{align*}
By the dominated convergence theorem, we may interchange limit and integration to get
\begin{align*}
    \Ee^x\biggl(u(X_{t\wedge \tau_{_{B(x,r)}}})-u(x)\biggr)
    &= \Ee^x\biggl(\int_{\big(0,\: t\wedge \tau_{_{B(x,r)}}\big)}\lim_{n\rightarrow\infty}Lu^x_n(X_{s})\,ds\biggr)\\
    &=\Ee^x\biggl(\int_{\big(0,\: t\wedge \tau_{_{B(x,r)}}\big)}Lu(X_{s})\,ds\biggr).
\end{align*}
Therefore, for any $x\in\R^d$ and $r>0$,
$$
    \Ee^x\biggl(u(X_{t\wedge \tau_{_{B(x,r)}}})-\int_0^{t\wedge \tau_{_{B(x,r)}}}Lu(X_{s})\,ds\biggr)
    =u(x).
$$

Because of \eqref{assumption}, the process $(X_t)_{t\geq 0}$ is conservative. Therefore, $\tau_{_{B(x,r)}}\xrightarrow{r\to\infty} \infty$, and we find by dominated convergence for all $u\in C_b^\infty(\R^d)$ that
$$
    \Ee^x\biggl(u(X_t)-\int_0^tLu(X_{s})\,ds\biggr)=u(x).
$$ Note that $(L, D(L))$ is the unique extension of $(A,D(A))$, $C_b^\infty (\R^d)\subset D(L)$. Now the Markov property shows that $u\in\widetilde{D}(A)$.

If we set $u(x)=e_{\xi}(x)$ and use that $Le_{\xi}(x)=Ae_{\xi}(x)$, the assertion follows.
\end{proof}

If $C_c^\infty (\R^d) \subset {D} (A)$, the following result can be deduced from \cite[Lemmas 4.1 and Lemma 5.1]{S3}. Here we will present a simple proof of it by making use of the extended domain $\widetilde{D}(A)$ of the operator.

\begin{proposition}\label{appendix2}
    Let $(X_t)_{t\ge 0}$ be a Feller process with generator $(A,D(A))$ such that $C_c^\infty(\R^d)\subset D(A)$ and \eqref{assumption} holds. Let $\tau_{_{{B(x,r)}}}$ be the first exit time of the process from the open ball ${B(x,r)}$. Then, for any $x\in \R^d$ and $r, t>0$,
    \begin{align}\stepcounter{equation}
        \Pp^x(\tau_{_{B(x,r)}}\le t)
        &\le c \, t\sup_{|y-x|\le r} \; \sup_{|\xi|\le 1/r}|p(y,\xi)|\tag{\theequation a}\label{max1a}\\
        &\le c\, t\sup_{|\xi|\le 1/r} \; \sup_{z\in\R^d}|p(z,\xi)|\tag{\theequation b}\label{max1b}
    \end{align}
    with an absolute constant $c>0$.
\end{proposition}

\begin{proof}
Pick $u\in C_c^\infty(\R^d)$ such that $\supp u\subset B(0,1)$, $u(0)=1$ and $0\le u\le 1$. For $x\in\R^d$ and $r>0$, set $u^x_r(\cdot):=u((\cdot-x)/r)$. Clearly, $u^x_r\in  C_c^\infty(\R^d)$. By \eqref{assumption} and since $C_c^\infty(\R^d)\subset {D}(A) \subset \widetilde{D}(A)$, $(M_t,\mathscr F_t)_{t\geq 0}$ is a martingale under $\Pp^x$, where
$$
    M_t:=1-u^x_r(X_{t\wedge \tau_{_{B(x,r)}}})+\int_0^{{t\wedge \tau_{B(x,r)}}}(-p(X_s,D))\,u_r^x(X_s)\,ds,
$$
and $\mathscr F_t = \sigma(X_s\::\: s\leq t)$ is the canonical filtration for $(X_t)_{t\geq 0}$.
Therefore,
$$
    \Ee^x\Bigl(1-u^x_r(X_{t\wedge \tau_{_{B(x,r)}}})\Bigr)
    =\Ee^x\biggl(\int_0^{{t\wedge \tau_{B(x,r)}}}p(X_s,D)\,u_r^x(X_s)\,ds\biggr)
$$
where $p(X_s,D)\,u_r^x(X_s)$ is short for $p(y,D_y)\,u_r^x(y)\Big|_{y=X_s}$. Now
\begin{equation}\label{proof1}\begin{aligned}
    \Pp^x&(\tau_{_{B(x,r)}}\le t)\\
    &\le \Ee^x\Bigl(1-u^x_r\bigl(X_{t\wedge \tau_{_{B(x,r)}}}\bigr)\Bigr)\\
    &=\Ee^x\biggl(\int_{\big(0,\: t\wedge \tau_{_{B(x,r)}}\big)}p(X_{s},D)\,u_r^x(X_{s})\,ds\biggr)\\
    &=\Ee^x\biggl(\int_{\big(0,\: t\wedge \tau_{_{B(x,r)}}\big)}\I_{\{|X_{s}-x|<r\}}\,p(X_{s},D)\,u_r^x(X_{s})\,ds\biggr)\\
    &=\Ee^x\biggl(\int_{\big(0,\: t\wedge \tau_{_{B(x,r)}}\big)}\biggl[\I_{\{|y-x|<r\}}\int e^{i\langle y,\xi\rangle}\,p(y,\xi)\,\hat{u}_r^x(\xi)\,d\xi\biggr]\bigg|_{y=X_{s}}\,ds\biggr)\\
    &\le \Ee^x\biggl(\int_{\big(0,\: t\wedge \tau_{_{B(x,r)}}\big)} \sup_{|z-x|<r}\int|p(z,\xi)||\hat{u}_r^x(\xi)|\,d\xi\,ds\biggr)\\
    &= \Ee^x\big({t\wedge \tau_{ B(x,r)}}\big)\sup_{|z-x|<r}\int|p(z,\xi)||\hat{u}_r^x(\xi)|\,d\xi\\
    &\le t \int\sup_{|z-x|<r}|p(z,\xi)|\,r^d\,|\hat{u}(r\xi)|\,d\xi\\
    &=t\int \sup_{|y-x|<r}|p(y,\xi/r)| |\hat{u}(\xi)|\,d\xi.
\end{aligned}\end{equation}
Thus we can use \cite[Lemma 2.3]{S1} to obtain
\begin{align*}
    \int \sup_{|y-x|\le r}|p(y,\xi/r)||\hat{u}(\xi)|\,d\xi
    &\le 2\sup_{|y-x|\le r}\;\sup_{|\eta|\le 1/r}|p(y,\eta)|\int(1+|\xi|^2)|\hat{u}(\xi)|\,d\xi\\
    &=: c_u\sup_{|y-x|\le  r}\sup_{|\eta|\le 1/r}|p(y,\eta)|,
\end{align*}
where $c_u=\int(1+|\xi|^2)|\hat{u}(\xi)|\,d\xi$. Combining all inequalities we get
\begin{equation*}\label{nonexp}
    \Pp^x(\tau_{_{B(x,r)}}\le t)
    \le c_u \,t\sup_{|y-x|\le r}\;\sup_{|\xi|\le 1/r}|p(y,\xi)|,
\end{equation*}
and the assertion follows.
\end{proof}

\medskip

Let $(X_t)_{t\ge0}$ be a strong Markov process with semigroup $(T_t)_{t\ge0}$ and generator $A$.  Assume that the semigroup $(T_t)_{t\ge 0}$ has the $C_b$-Feller property, i.e.\ $T_t(C_b(\R^d))\subset C_b(\R^d)$ for all $t>0$, where $C_b(\R^d)$ is the set of bounded and continuous functions on $\R^d$. Moreover, we assume that $C_c^\infty(\R^d)$ is contained in the extended domain $\widetilde D(A)$ of the operator $A$ and that $A|_{C_c^\infty(\R^d)} =-p(\cdot,D)$ where $-p(\cdot,D)$ is a pseudo differential operator with symbol $p(x,\xi)$. Then, for any $u\in C_c^\infty(\R^d)$, $$
    \biggl(u(X_t)-\int_0^t(-p(X_s,D)u(X_s))\,ds,\,\,\mathscr{F}_t\biggr)_{t\geq 0}
    \quad\text{is a local martingale under\ \ } \Pp^x.
$$
Furthermore, we have the following simple condition on the symbol $p(x,\xi)$ to yield the ($C_\infty$-)Feller property of $(T_t)_{t\ge0}$.

\begin{proposition}\label{appendix3}
If the symbol $p(x,\xi)$ satisfies
\begin{equation}\label{fess}
    \varlimsup_{r\rightarrow\infty}\sup_{|x|\le r}\sup_{|\xi|\le 1/r}|p(x,\xi)|=0,
\end{equation}
then $(T_t)_{t\ge 0}$ has the Feller property, i.e.\ $T_t(C_\infty(\R^d))\subset C_\infty(\R^d)$ for every $t\ge0$ where $C_\infty(\R^d)$ is the set of continuous functions on $\R^d$ vanishing at infinity.
\end{proposition}

\begin{proof}
Since $\sqrt{|p(x,\cdot)|}$ is, for any fixed $x\in\R^d$, subadditive it is not hard to see that \eqref{fess} is equivalent to \begin{equation}\label{fes}
    \varlimsup_{r\rightarrow\infty}\sup_{|x|\le\gamma r}\sup_{|\xi|\le 1/r}|p(x,\xi)|=0,\qquad\gamma\ge 1.
\end{equation}

A close inspection of the proof of Proposition \ref{appendix2} shows that \eqref{max1a} also holds in the present setting. For every $f\in C_\infty(\R^d)$ we see by the $C_b$-Feller property that $T_tf\in C_b(\R^d)$ is continuous. We have to study the behaviour of $T_tf(x)$ as $|x|\to\infty$. If $f\in C_\infty(\R^d)$, we find for every $\varepsilon>0$ some $r_1:=r_1(\varepsilon,f)>0$ such that
$$
    |f(y)|\le \varepsilon/2\quad\text{for all\ \ } |y|\ge r_1.
$$
Because of \eqref{fes}, there is some constant $r_2:=r_2(\varepsilon,f)>r_1 > 0$ such that
$$
    \sup_{|z|\le 3|y|/2}\sup_{|\xi|\le 2/|y|}|p(z,\xi)|
    \le
    \frac{\varepsilon}{2c\,t(\|f\|_\infty+1)}
    \quad\text{for all\ \ } |y|\ge r_2
$$
($c$ is the constant appearing in Proposition \ref{appendix2}). By \eqref{max1a} we find for $y\in\R^d$ with $|y|\ge 2r_2$
\begin{align*}
    |(T_tf)(y)|
    &\le \int |f(z)|\,\Pp^y(X_t\in dz)\\
    &=\int_{B(0,r_2)} |f(z)|\,\Pp^y(X_t\in dz)+\int_{B^c(0,r_2)} |f(z)|\,\Pp^y(X_t\in dz)\\
    &\le \|f\|_\infty \,\Pp^y(|X_t|\le r_2 )+\varepsilon/2\\
    &\le  \|f\|_\infty \,\Pp^y(|X_t-y|\ge |y|- r_2 )+\varepsilon/2\\
    &\le  \|f\|_\infty \,\Pp^y(\sup_{s\le t}|X_s-y|\ge |y|/2)+\varepsilon/2\\
    &\le  c\,t\,\|f\|_\infty \sup_{|z-y|\le |y|/2} \; \sup_{|\xi|\le 2/|y|}|p(z,\xi)|+\varepsilon/2\\
    &\le  c\,t\,\|f\|_\infty \sup_{|z|\le 3|y|/2}\;\sup_{|\xi|\le 2/|y|}|p(z,\xi)|+\varepsilon/2\\
    &\le \varepsilon,
\end{align*}
which shows that $\lim_{|y|\to\infty} T_t f(y) = 0$ for all $f\in C_\infty(\R^d)$.
\end{proof}

The following statement presents a general connection between a $C_b$-Feller and a ($C_\infty$-)Feller semigroup.

\begin{proposition}\label{appendix5} \noindent
\textup{\bfseries (i)} Suppose that $(T_t)_{t\ge 0}$ is a Feller semigroup. If $T_t1\in C_b(\R^d)$ for every fixed $t\ge 0$, then $(T_t)_{t\ge 0}$ is a $C_b$-Feller semigroup. In particular, any conservative Feller semigroup, i.e.\ for $t\ge0$, $T_t1=1$, is a $C_b$-Feller semigroup.

\smallskip\noindent
\textup{\bfseries (ii)}
Let $(T_t)_{t\ge 0}$ be a $C_b$-Feller semigroup and $(P(t,x,dy))_{t>0}$ the corresponding family of kernels, i.e.\ for any $t>0$, $x\in\R^d$ and $u\in C_b(\R^d)$, $T_tu(x)=\int u(y) \, P(t,x,dy)$. Then, $(T_t)_{t>0}$ is a Feller semigroup if, and only if, for all $t>0$ and all bounded sets $B\in \mathscr{B}(\R^d)$,
$$
    \lim_{|x|\rightarrow\infty}P(t,x,B)=0.
$$
\end{proposition}

\begin{proof}\noindent
\textup{\bfseries (i)} This is just \cite[Corollary 3.4]{S1}.

\smallskip\noindent
\textup{\bfseries (ii)}
Assume that $(T_t)_{t\ge0}$ is $C_b$-Feller. Then, for any $t>0$ and $f\in C_\infty(\R^d)$, $T_tf$ is continuous. For any $\varepsilon>0$, we first choose $\delta>0$ such that $|f|\I_{B(0,\delta)^c}\le \varepsilon$. Thus, for $x\in\R^d$,
\begin{align*}
    |T_tf(x)|
    &\le \int_{B(0,\delta)}|f(y)|\,P(t,x,dy)+\int_{B(0,\delta)^c}|f(y)|\,P(t,x,dy)\\
    &\le \|f\|_\infty P(t,x, B(0,\delta))+\varepsilon.
\end{align*}
Hence,
$$
    \lim_{|x|\rightarrow\infty}|T_tf(x)|
    \le\|f\|_\infty \lim_{|x|\rightarrow\infty}P(t,x, B(0,\delta))+\varepsilon=\varepsilon.
$$
Letting $\varepsilon\rightarrow0$ yields that $T_tf\in C_\infty (\R^d)$.

On the other hand, for any bounded set $B\in \mathscr{B}(\R^d)$, we can choose some $f\in C_\infty(\R^d)$ such that $f\ge0$ and $f|_B\equiv 1$. Therefore,
$$
    T_tf(x)
    \ge \int_B f(y)\,P(t,x,dy)
    =P(t,x,B).
$$
Since $(T_t)_{t\ge 0}$ is $(C_\infty$-)Feller,
\begin{gather*}
    0
    =\lim_{|x|\rightarrow\infty}|T_tf(x)|
    =\lim_{|x|\rightarrow\infty}T_tf(x)
    \ge\lim_{|x|\rightarrow\infty}P(t,x,B).
    \qedhere
\end{gather*}
\end{proof}

We close this section with an abstract result for Feller semigroups.
\begin{proposition}\label{appendix4}
    The martingale problem for $(-p(\cdot,D), C_c^\infty(\R^d))$ is well posed if, and only if, the test functions $C_c^\infty(\R^d)$ are an operator core for the Feller generator $(A,D(A))$, i.e.\ $\overline{A|_{C_c^\infty(\R^d)}}=A$.
\end{proposition}

\begin{proof}
Assume that the martingale problem for $(-p(\cdot,D), C_c^\infty(\R^d))$ is well posed. According to a result by van Casteren, \cite[Theorem 2.5, Page 283]{vancas1}, see also that by Okitaloshima and van Casteren, \cite[Theorem 3.1, Page 789]{vancas2}, there exists a unique extension $(A,D(A))$ of $(-p(\cdot,D), C_c^\infty(\R^d))$ which is a Feller generator. In particular, $\overline{A|_{C_c^\infty(\R^d)}}=A$.

On the other hand, suppose that the test functions $C_c^\infty(\R^d)$ are an operator core for the Feller operator $(A,D(A))$. By the Hille-Yosida-Ray Theorem, see e.g.\ \cite[Chapter 4, Theorem 2.2, Page 165]{EK}, the range $(\lambda+p(\cdot,D))(C_c^\infty(\R^d))$ is dense in $C_\infty(\R^d)$ for some $\lambda>0$. Since $(-p(\cdot,D), C_c^\infty(\R^d))$ satisfies the positive maximum principle, it is dissipative in the sense that
$$
    \|\lambda u- (-p(\cdot,D))u\|_\infty
    \ge \lambda\|u\|_\infty
    \qquad\text{for all\ \ } u\in C_c^\infty(\R^d),
$$
cf.\ \cite[Chapter 4, Theorem 2.1, Page 165]{EK}.  Therefore, the well-posedness of the martingale problem for $(-p(\cdot,D), C_c^\infty(\R^d))$  follows from \cite[Chapter 4, Theorem 4.1, Page 182]{EK}.
 \end{proof}
\begin{ack}
Financial support through DFG (grant Schi 419/5-2) and DAAD (PPP Kroatien) (for R.L.\ Schilling),
the Alexander-von-Humboldt Foundation and the Programme of Excellent Young Talents in Universities
of Fujian (No.\ JA10058 and JA11051) (for Jian Wang)
is gratefully
acknowledged. Most of this work was done when Jian Wang was a Humboldt fellow at TU Dresden. He is grateful for the hospitality and the good working conditions.
\end{ack}

\end{document}